\documentclass[12pt]{amsart}
\usepackage[dvips]{graphicx,color}
\usepackage[dvips]{epsfig}
\usepackage{amsmath,amssymb,latexsym,amscd,xypic}

\DeclareRobustCommand{\SkipTocEntry}[4]{}
\makeatletter
\newcommand\@dotsep{4.5}
\def\@tocline#1#2#3#4#5#6#7{\relax
  \ifnum #1>\c@tocdepth 
  \else
    \par \addpenalty\@secpenalty\addvspace{#2}%
    \begingroup \hyphenpenalty\@M
    \@ifempty{#4}{%
      \@tempdima\csname r@tocindent\number#1\endcsname\relax
    }{%
      \@tempdima#4\relax
    }%
    \parindent\z@ \leftskip#3\relax \advance\leftskip\@tempdima\relax
    \rightskip\@pnumwidth plus1em \parfillskip-\@pnumwidth
    #5\leavevmode\hskip-\@tempdima #6\relax
    \leaders\hbox{$\m@th
      \mkern \@dotsep mu\hbox{.}\mkern \@dotsep mu$}\hfill
    \hbox to\@pnumwidth{\@tocpagenum{#7}}\par
    \nobreak
    \endgroup
  \fi}
\makeatother 

\DeclareFontFamily{OT1}{rsfs}{}
\DeclareFontShape{OT1}{rsfs}{n}{it}{<-> rsfs10}{}
\DeclareMathAlphabet{\curly}{OT1}{rsfs}{n}{it}

\newcommand{\cO}{\mathcal{O}}

\newcommand{\T}{\mathbf{T}}
\newcommand\C{\mathbb C}
\newcommand\CC{\mathsf C}
\newcommand\F{\curly F}
\newcommand{\Ib}{{\mathbb{I}}}
\newcommand\I{\curly I}

\newcommand\m{\mathfrak m}

\newcommand\OO{\mathcal{O}}
\newcommand\PP{\mathbb P}
\newcommand\Pp{{\mathbb P}^1}
\newcommand\FF{\mathbb F}
\newcommand\FFF{\mathsf{F}}
\newcommand\GGG{\mathsf{G}}

\newcommand\Q{\mathbb Q}

\newcommand\Z{\mathbb Z}
\newcommand\bW{\mathsf{W}}

\newcommand\ZZ{\mathsf Z}
\newcommand\com{\mathbb C}

\newcommand{\bV}{\mathsf{V}}

\newcommand{\bw}{\mathsf{w}}

\newcommand{\rarr}{\rightarrow}

\newcommand{\Rt}[1]{\stackrel{#1\,}{\longrightarrow}}
\newcommand\To{\longrightarrow}
\newcommand\into{\hookrightarrow}
\newcommand\Into{\ar@{^{ (}->}[r]}

\newcommand\bull{{\scriptscriptstyle\bullet}}
\newcommand\udot{^\bull}

\newcommand\rk{\operatorname{rank}}
\newcommand\tr{\operatorname{tr}}

\newcommand\Hom{\operatorname{Hom}}
\renewcommand\hom{\curly H\!om}
\newcommand\Ext{\operatorname{Ext}}

\newcommand{\SSS}{\mathcal{S}}
\DeclareMathOperator{\Ver}{Vert}
\DeclareMathOperator{\Sym}{Sym}
\DeclareMathOperator{\sgn}{sgn}
\DeclareMathOperator{\ch}{ch}

\newcommand\beq{\begin{equation}}
\newcommand\eeq{\end{equation}}

\newtheorem{thm}{Theorem}

\newtheorem{lem}{Lemma}

\newtheorem{conj}{Conjecture}
\newtheorem{prop}{Proposition}

\title{\textbf{Descendents on local curves: Rationality}}
\author{R. Pandharipande and A. Pixton}
\date{May 2012}

\begin{document}

\begin{abstract} \noindent
We study the stable pairs theory of local curves in 3-folds  
with descendent insertions.
The rationality of the partition function
of descendent invariants is established
for the full local curve geometry (equivariant
with respect
to the scaling 2-torus) including relative
conditions and odd degree insertions  for higher genus curves. 
The capped 1-leg descendent vertex (equivariant
with respect to the 3-torus)
is also proven to be rational. 
The results are obtained by combining geometric constraints
with a detailed analysis of the poles of the descendent
vertex.

\end{abstract}

\maketitle

\setcounter{tocdepth}{1} 
\tableofcontents


\setcounter{section}{-1}
\section{Introduction}

\subsection{Descendents}\label{dess}
Let $X$ be a nonsingular 3-fold, and let
$$\beta \in H_2(X,\mathbb{Z})$$ be a nonzero class. We will study here the
moduli space of stable pairs
$$[\OO_X \stackrel{s}{\rightarrow} F] \in P_n(X,\beta)$$
where $F$ is a pure sheaf supported on a Cohen-Macaulay subcurve of $X$, 
$s$ is a morphism with 0-dimensional cokernel, and
$$\chi(F)=n, \  \  \ [F]=\beta.$$
The space $P_n(X,\beta)$
carries a virtual fundamental class obtained from the 
deformation theory of complexes in
the derived category \cite{pt}. A review can be found in Section \ref{ooo}.

Since $P_n(X,\beta)$ is a fine moduli space, there exists a universal sheaf
$$\FF \rightarrow X\times P_{n}(X,\beta),$$
see Section 2.3 of \cite{pt}.
For a stable pair $[\OO_X\to F]\in P_{n}(X,\beta)$, the restriction of
$\FF$
to the fiber
 $$X \times [\OO_X \to F] \subset 
X\times P_{n}(X,\beta)
$$
is canonically isomorphic to $F$.
Let
$$\pi_X\colon X\times P_{n}(X,\beta)\to X,$$
$$\pi_P\colon X\times P_{n}(X,\beta)
\to P_{n}(X,\beta)$$
 be the projections onto the first and second factors.
Since $X$ is nonsingular
and
$\FF$ is $\pi_P$-flat, $\FF$ has a finite resolution 
by locally free sheaves.
Hence, the Chern character of the universal sheaf $\FF$ on 
$X \times P_n(X,\beta)$ is well-defined.
By definition, the operation
$$
\pi_{P*}\big(\pi_X^*(\gamma)\cdot \text{ch}_{2+i}(\FF)
\cap(\pi_P^*(\ \cdot\ )\big)\colon 
H_*(P_{n}(X,\beta))\to H_*(P_{n}(X,\beta))
$$
is the action of the descendent $\tau_i(\gamma)$, where
$\gamma \in H^*(X,\Z)$.

For nonzero $\beta\in H_2(X,\Z)$ and arbitrary $\gamma_i\in H^*(X,\Z)$,
define the stable pairs invariant with descendent insertions by
\begin{eqnarray*}
\left\langle \prod_{j=1}^k \tau_{i_j}(\gamma_j)
\right\rangle_{\!n,\beta}^{\!X}&  = &
\int_{[P_{n}(X,\beta)]^{vir}}
\prod_{j=1}^k \tau_{i_j}(\gamma_j) \\
& = & 
\int_{P_n(X,\beta)} \prod_{j=1}^k \tau_{i_j}(\gamma_{j})
\Big( [P_{n}(X,\beta)]^{vir}\Big).
\end{eqnarray*}
The partition function is 
$$
\ZZ^X_{\beta}\left(   \prod_{j=1}^k \tau_{i_j}(\gamma_{j})
\right)
=\sum_{n} 
\left\langle \prod_{j=1}^k \tau_{i_j}(\gamma_{j}) 
\right\rangle_{\!n,\beta}^{\!X}q^n.
$$

Since $P_n(X,\beta)$ is empty for sufficiently negative
$n$, 
$\ZZ^X_{\beta}\big(   \prod_{j=1}^k \tau_{i_j}(\gamma_{j})
\big)$
is a Laurent series in $q$. The following conjecture was made in 
\cite{pt2}.

\begin{conj}
\label{111} 
The partition function
$\ZZ_{\beta}^X\big(   \prod_{j=1}^k \tau_{i_j}(\gamma_{j})
\big)$ is the 
Laurent expansion of a rational function in $q$.
\end{conj}

If only primary field insertions $\tau_0(\gamma)$ appear, 
Conjecture \ref{111} is known for
toric $X$ by \cite{moop, mpt} and for Calabi-Yau $X$ by
\cite{bridge,toda} together with \cite{joy}.
In the presence of descendents $\tau_{i>0}(\gamma)$,
very few results have been obtained.

The central result of the present paper is the proof of
Conjecture 1 in case $X$ is the total space of
an rank 2 bundle over a curve, a {\em local curve}. 
In fact, the rationality
of the stable pairs descendent theory of relative local curves
is proven. 

\subsection{Local curves} \label{lc1}
Let $N$ be a split 
rank 2 bundle on a nonsingular projective curve $C$ of genus $g$,
\begin{equation}\label{ffg}
N=L_1\oplus L_2.
\end{equation}
The splitting determines a scaling action of a 2-dimensional torus
$$T=\C^* \times \C^*$$ on $N$.
The {\em level} of the splitting is the pair of integers
$(k_1,k_2)$ where,
$$k_i= {\text {deg}}(L_i).$$
Of course, the scaling action and the level
depend upon the choice of splitting \eqref{ffg}.

Let $s_1,s_2 \in H^*_\T(\bullet)$ be the first Chern classes
of the standard representations of the first and second
$\C^*$-factors of $T$ respectively.
We define
\begin{equation}\label{lwww}
\left\langle \prod_{j=1}^k \tau_{i_j}(\gamma_j)
\right\rangle_{\!n,d}^{\!N}  = 
\int_{[P_{n}(N,d)]^{vir}}
\prod_{j=1}^k \tau_{i_j}(\gamma_j) \ \ \ \in \mathbb{Q}(s_1,s_2)\ .
\end{equation}
Here,
the curve class is $d$ times the zero section $C \subset N$ and
$$\gamma_j \in H^*(C,\mathbb{Z})\ .$$
The right side of \eqref{lwww} is defined by $T$-equivariant
residues as in \cite{BryanP,lcdt}. Let
$$
\ZZ_{d}^{N}\left(   \prod_{j=1}^k \tau_{i_j}(\gamma_{j})
\right)^T
=\sum_{n}
\left\langle \prod_{j=1}^k \tau_{i_j}(\gamma_{j}) 
\right\rangle_{\!n,d}^{\! N}q^n.
$$

\begin{thm}
\label{onnn} 
$\ZZ_{d}^{N}\big(   \prod_{j=1}^k \tau_{i_j}(\gamma_{j})
\big)^T$ is the 
Laurent expansion in $q$ of a rational function in $\mathbb{Q}(q,s_1,s_2)$.
\end{thm}

The rationality of Theorem \ref{onnn} holds even when
$\gamma_j \in H^1(C,\mathbb{Z})$.
Theorem \ref{onnn} is proven via
the stable pairs theory of relative
local curves and the 1-leg descendent vertex.
The proof provides a  method to
compute $\ZZ_{d}^{N}\big(   \prod_{j=1}^k \tau_{i_j}(\gamma_{j})
\big)^T$.
\subsection{Relative local curves} \label{lc2}

\label{relgeom}
The fiber of $N$ over a point $p\in C$ determines a $T$-invariant
divisor 
$$N_p \subset N$$
isomorphic to $\com^2$ with the standard $T$-action.
For $r>0$, we will consider the local theory of $N$
relative to the divisor
$$S= \bigcup_{i=1}^r N_{p_i} \subset N$$
determined by the fibers over $p_1,\ldots,p_r\in C$.
Let $P_n(N/S,d)$ denote the relative moduli space of stable pairs, see 
\cite{pt} for a discussion.

For each $p_i$, let $\eta^i$
be a partition of $d$ weighted
by the equivariant Chow ring, 
$$A_T^*(N_{p_i},{\mathbb Q})\stackrel{\sim}{=} {\mathbb Q}[s_1,s_2],$$
of the fiber $N_{p_i}$.
By Nakajima's construction,
a weighted partition $\eta^i$ determines a $T$-equivariant class 
$$\CC_{\eta^i} \in A_T^*(\text{Hilb}(N_{p_i},d), \mathbb{Q})$$
in the
Chow ring of the Hilbert scheme of points.
In the theory of stable pairs, the weighted partition $\eta^i$
specifies relative
conditions via the boundary map
$$\epsilon_i: P_n(N/S,d)\rightarrow \text{Hilb}(N_{p_i},d).$$

An element $\eta\in {\mathcal P}(d)$ of the set of 
partitions of $d$ may be 
viewed as a 
weighted partition with all weights set to the identity class
$$1\in H^*_T(N_{p_i},{\mathbb Q})\ .$$
The Nakajima basis of $A_T^*(\text{Hilb}(N_{p_i},d), \mathbb{Q})$ consists of 
identity weighted partitions indexed by ${\mathcal P}(d)$. 
The $T$-equivariant intersection pairing in the Nakajima basis is
$$g_{\mu\nu}=\int_{\text{Hilb}(N_{p_i},d)} \CC_\mu \cup \CC_\nu =
\frac{1}{(s_1s_2)^{\ell(\mu)}} 
\frac{(-1)^{d-\ell(\mu)}} 
{{\mathfrak{z}}(\mu)}\ {\delta_{\mu,\nu}},$$
where
$${\mathfrak z}(\mu) =  \prod_{i=1}^{\ell(\mu)} \mu_i \cdot 
|\text{Aut}(\mu)|.$$
Let $g^{\mu\nu}$ be the inverse matrix.

The notation $\eta([0])$ will be used to set all
weights to $[0]\in A^*_T(N_{p_i},{\mathbb Q} )$.
Since
$$[0]= s_1s_2 \in A^*_T(N_{p_i}, {\mathbb Q} ),$$
the weight choice has only a mild effect.

Following the
 notation of \cite{BryanP,lcdt},
the relative stable pairs partition function with
descendents,
\begin{equation*}
{\mathsf Z}^{N/S}_{d,\eta^1,\dots,\eta^r} 
\left(   \prod_{j=1}^k \tau_{i_j}(\gamma_{j})
\right)^T
 =\sum _{n\in \Z }q^{n}
\int _{[P_{n} (N/S,d)]^{vir}}
  \prod_{j=1}^k \tau_{i_j}(\gamma_{j})\ 
\prod_{i=1}^r \epsilon_i^*(\CC_{\eta^i}),
\end{equation*}
is well-defined for local curves.

\begin{thm}
\label{tnnn} 
$\ZZ_{d,\eta^1,\dots,\eta^r}
^{N/S}\big(   \prod_{j=1}^k \tau_{i_j}(\gamma_{j})
\big)^T$ is the 
Laurent expansion in $q$ of a rational function in $\mathbb{Q}(q,s_1,s_2)$.
\end{thm}

Theorem \ref{tnnn} implies Theorem \ref{onnn} by the degeneration formula.
The
proof of Theorem \ref{tnnn} uses the TQFT formalism exploited in
\cite{BryanP,lcdt} together with an analysis of
the capped 1-leg descendent vertex.

\subsection{Capped 1-leg descendent vertex} \label{legger}
The capped 1-leg geometry concerns 
the trivial bundle,
$$N = \cO_{\PP^1} \oplus \cO_{\PP^1} \rightarrow \PP^1\ ,$$ 
relative to the fiber
$$N_\infty \subset N$$
over $\infty \in \PP^1$.
Capped geometries have been studied (without descendents)
in \cite{moop}.

The total space $N$ naturally carries an action of a 
3-dimensional torus $$\mathbf{T} = T \times \com^*\ .$$
Here, $T$ acts as before by  scaling the
factors of $N$ and preserving the relative divisor $N_\infty$. 
The $\com^*$-action
on the base $\PP^1$ which fixes the points  $0, \infty\in \PP^1$ 
lifts to an additional $\com^*$-action on $N$ fixing
$N_\infty$.

The equivariant cohomology 
ring $H_{\mathbf{T}}^*(\bullet)$ is generated by
the Chern classes $s_1$, $s_2$, and $s_3$
of the standard representation of the three $\com^*$-factors.
We define 
\begin{equation}\label{pppw}
{\mathsf Z}^{\mathsf{cap}}_{d,\eta} 
\left(   \prod_{j=1}^k \tau_{i_j}(\gamma_{j})
\right)^{\mathbf{T}}
 =\sum _{n\in \Z }q^{n}
\int _{[P_{n} (N/N_\infty,d)]^{vir}}
  \prod_{j=1}^k \tau_{i_j}(\gamma_{j})\ 
\cup \epsilon_\infty^*(\mathsf{C}_{\eta}),
\end{equation}
by $\mathbf{T}$-equivariant residues.\footnote{The $T$-equivariant
series associated to the cap will be denoted  
$${\mathsf Z}^{\mathsf{cap}}_{d,\eta} 
\left(   \prod_{j=1}^k \tau_{i_j}(\gamma_{j})
\right)^T \ ,
$$
for $\gamma_j\in H^*(\Pp,\mathbb{Z})$.}
Here, $\gamma_j \in H^*_{\mathbf{T}}(\PP^1,\mathbb{Z})$.
By definition, the partition function \eqref{pppw} is
a Laurent series in $q$ with coefficients in the field
$\mathbb{Q}(s_1,s_2,s_3)$.

\begin{thm}
\label{cnnn} 
$
{\mathsf Z}^{\mathsf{cap}}_{d,\eta} 
\left(   \prod_{j=1}^k \tau_{i_j}(\gamma_{j})
\right)^{\mathbf{T}}$
is the 
Laurent expansion in $q$ of a rational function in 
$\mathbb{Q}(q,s_1,s_2,s_3)$.
\end{thm}

Theorem \ref{cnnn} is the main contribution of the paper. 
The result relies upon
a delicate cancellation of
poles in the vertex formula of \cite{pt2}
for stable pairs invariants.
Theorem \ref{tnnn} is derived as a consequence.

\subsection{Stationary theory}
In \cite{parttwo}, we prove reduction rules for
stationary descendents in the 
$T$-equivariant local theory of curves. Let 
$\mathsf{p}\in H^2(C,\mathbb{Z})$ be the class of a point on a
nonsingular curve $C$. The stationary descendents are 
$\tau_i(\mathsf{p})$.
For the degree $d$ local theory of $C$, we
find universal formulas expressing the descendents 
$\tau_{i>d}(\mathsf{p})$ in terms of the descendents
$\tau_{i\leq d}(\mathsf{p})$.
The reduction rules provide an alternative (and more effective)
approach 
 to
the rationality of Theorem \ref{tnnn} in the 
stationary case.

The exact calculation in \cite{parttwo} of the basic stationary
descendent series
$$\mathsf{Z}^{\mathsf{cap}}_{d,(d)}( \tau_d(\mathsf{p}))^T =
\frac{q^d}{d!}\left(\frac{s_1+s_2}{s_1s_2}\right)
\frac{1}{2}\sum_{i=1}^d  \frac{ 1+(-q)^{i}}{1-(-q)^i} \ $$
plays a special role.
The coefficient of $q^d$, 
$$ \left\langle \tau_d, (d) \right\rangle_{\text{Hilb}(\com^2,d)}=
\frac{1}{2\cdot (d-1)!} \left(\frac{s_1+s_2}{s_1s_2}\right),$$
is the classical $T$-equivariant pairing on the
Hilbert scheme of $d$ points in $\C^2$.

The $T$-equivariant stationary descendent theory is simpler
than the full descendent theories studied here. We do not know an
alternative approach to the rationality of the
full $T$-equivariant descendent theory
of local curves. Even the rationality of the  $\mathbf{T}$-equivariant
stationary theory of the cap does not appear to
be accessible via \cite{parttwo}.

The methods of \cite{parttwo} also prove a functional equation for 
the partition function for stationary descendents which is 
a special case of the following conjecture we make here.

\begin{conj}
\label{33345} 
Let $\ZZ_{d,\eta^1,\dots,\eta^r}
^{N/S}\big(   \prod_{j=1}^k \tau_{i_j}(\gamma_{j})
\big)^T$ be the Laurent expansion in $q$ of 
$F(q,s_1,s_2) \in \mathbb{Q}(q,s_1,s_2)$. Then, $F$
satisfies the functional equation
\[
F(q^{-1},s_2,s_2) = (-1)^{\Delta+|\eta|-\ell(\eta)
+
\sum_{j=1}^k
i_j}q^{-\Delta}F(q,s_1,s_2),
\]
where the constants are defined by 
$$\Delta = \int_{\beta}c_1(T_N),\ \ \ 
|\eta|=\sum_{i=1}^r |\eta^i|,\ \ \  \text{and} \ \ \
\ell(\eta)=\sum_{i=1}^r \ell(\eta^i) 
\ .$$
\end{conj}

Here, $T_N$ is the tangent bundle of the 3-fold $N$, and
$\beta$ is the curve class given by $d$ times the $0$-section.
We believe the straightforward generalization of Conjecture \ref{33345}
to all descendent partition functions for the stable
pairs theories of relative 3-folds (equivariant and non-equivariant) holds.
If there are no descendents, the functional equation is known
to hold in the toric case \cite{moop}.
The strongest evidence with descendents
  is the stationary result of Theorem 2
of \cite{parttwo}.

\subsection{Denominators}

The descendent partition 
functions for the stable pairs theory of local
curves have very restricted
denominators when considered as rational functions in $q$
with coefficients in $\Q(s_1,s_2)$ for Theorems \ref{onnn}-\ref{tnnn}
and rational functions in $q$
with coefficients in $\Q(s_1,s_2,s_3)$ for Theorem \ref{cnnn}.
\begin{conj}
\label{222} 
The denominators of the degree $d$ descendent partition functions
$\ZZ$ of Theorems \ref{onnn}, \ref{tnnn}, and \ref{cnnn}
are  products of factors of the form $q^k$ and
$$1-(-q)^r$$
for $1\leq r \leq d$.
\end{conj}

In other words, the poles in $-q$ are conjectured to occur
only at 0 and
 $r^{th}$ roots for $r$ at most $d$ (and have no
dependence on the variables $s_i$). 
Conjecture \ref{222} is proven in Theorem \ref{2222} of
 Section \ref{ennd}
for descendents
of even cohomology. 
The denominator restriction yields new
results about the 3-point functions of the Hilbert scheme of
points of $\mathbb{C}^2$ stated as a Corollary to Theorem \ref{2222}.

\subsection{Descendent theory of toric 3-folds}
Calculation of the descendent theory of stable pairs on  nonsingular
toric 3-folds requires
 knowledge of the capped 3-leg descendent vertex.{\footnote{The
capped 2-leg descendent vertex is, of course, a specialization
of the 3-leg vertex.}}
The rationality of the capped 3-leg descendent vertex is proven
in \cite{part3} via a geometric reduction to the 1-leg case
of Theorem \ref{cnnn}. 
As a result, Conjecture \ref{111} is
established for all nonsingular toric 3-folds. The rationality
of the descendent theory of
several log Calabi-Yau geometries 
is also proven in \cite{part3}.

\subsection{Plan of the paper}

After a brief review of the theory of 
stable pairs in Section \ref{ooo}, the
vertex formalism of \cite{pt2} is summarized in
Section \ref{ttt}. 
The proof of Theorem \ref{cnnn} is presented in Section \ref{333}
for descendents of the nonrelative $\mathbf{T}$-fixed point 
$0\in \PP^1$ modulo the
pole cancellation property established
in Section \ref{polecan}.
Depth and the rubber calculus for stable pairs of local curves
are discussed in Sections \ref{depp} and \ref{rubc}.
The full statement of Theorem \ref{cnnn}
 is obtained in  Section
\ref{444}. In fact,
the rationality of the $\mathbf{T}$-equivariant descendent
theories of all twisted caps and tubes is established
in Section \ref{444}.
Theorems \ref{onnn} and \ref{tnnn} are proven as a 
consequence of Theorem \ref{cnnn} in Section \ref{555} using the 
methods of \cite{BryanP,vir,lcdt}. Denominators
are studied in Section \ref{ennd}.

\subsection{Other directions}
Whether parallel results can be obtained for the
local Gromov-Witten theory of curves \cite{BryanP} is
an interesting question. Although conjectured to be equivalent,
the descendent theory of stable pairs on $3$-folds appears more accessible
than descendents in Gromov-Witten theory. The direct vertex analysis undertaken here for
Theorem \ref{cnnn} must be replaced 
in Gromov-Witten theory  with a deeper
understanding of Hodge integrals \cite{FP}.

Another advantage of stable pairs, at least for Calabi-Yau
geometries, is the possibility of using 
motivic integrals with respect to Beh\-rend's $\chi$-function \cite{Beh},
see \cite{pt3} for an early use. Recently, D. Maulik and R. P. Thomas
have been pursuing $\chi$-functions in the log Calabi-Yau setting.
Applications to the rationality of descendent series in
Fano geometries might be possible.

A principal motivation of studying descendents for stable
pairs is the perspective of \cite{mptop}.  
Descendents constrain relative invariants. With the
degeneration formula, the possibility emerges of
studying stable pairs on arbitrary (non-toric) 3-folds.

\subsection{Acknowledgements}
Discussions with J. Bryan, D. Maulik, A. Oblomkov, A. Okounkov, and
R.~P. Thomas 
about the stable pairs vertex, self-dual obstruction theories,
and rationality
played an important role. We thank M. Bhargava and M. Haiman
for conversations related to the pole cancellation of
Section \ref{polecan}.
The study of descendents
for 3-fold sheaf theories in \cite{MNOP2,pt2} 
motivated several 
aspects of the paper. 
 
R.P. was partially supported by NSF grants DMS-0500187
and DMS-1001154.
A.P. was supported by a NDSEG graduate fellowship.
The paper was completed in the summer of 2010
while visiting the 
Instituto Superior T\'ecnico in Lisbon where
R.P. was supported by a Marie Curie fellowship and
a grant from the Gulbenkian foundation.


\section{Stable pairs on $3$-folds}
\label{ooo}
\subsection{Definitions}
Let $X$ be a nonsingular quasi-projective $3$-fold over $\mathbb{C}$
with 
polarization $L$.
Let $\beta\in H_2(X,\mathbb{Z})$ be a nonzero class.
The moduli space $P_n(X,\beta)$ parameterizes \emph{stable pairs}
\begin{equation}\label{vqq2}
\OO_X \stackrel{s}{\rightarrow} F
\end{equation}
where $F$ is a sheaf with Hilbert polynomial
$$ \chi(F\otimes L^k) = k\int_\beta c_1(L) + n$$
 and $s\in H^0(X,F)$ is a section. 
The two stability conditions are: 
\begin{enumerate}
\item[(i)]
the sheaf $F$ is {pure} with proper support, 
\item[(ii)] the section $\OO_X \stackrel{s}{\rightarrow} F$ has 0-dimensional
cokernel.
\end{enumerate}
By definition, {\em purity} (i) means
every nonzero
subsheaf of $F$ has support of dimension 1 \cite{HLShaves}. In particular,
 purity implies  the (scheme-theoretic) 
support $C_F$ of $F$ is a Cohen-Macaulay curve. 
A quasi-projective moduli space of stable pairs
 can be constructed by  a standard GIT analysis of Quot scheme
quotients \cite{LPPairs1}.

For convenience, we will often refer to the stable pair 
\eqref{vqq2} on $X$ simply by $(F,s)$.

\subsection{Virtual class}
A central result of \cite{pt} is the construction of a
virtual class on $P_n(X,\beta)$.
The standard approach to the deformation theory of pairs
fails to yield an appropriate 2-term deformation theory
for $P_n(X,\beta)$.
Instead, $P_n(X,\beta)$ is viewed in \cite{pt}
as a moduli space
of complexes in the derived category.

Let $D^b(X)$ be the bounded derived category of coherent
sheaves on $X$.
Let
$${I}\udot = \left\{ \OO_X \rightarrow F \right\}\in
D^b(X)$$
be the complex determined by a stable pair.
The tangent-obstruction theory obtained by deforming ${I}\udot$
in $D^b(X)$ while fixing its determinant is 2-term and governed by the 
groups{\footnote{The subscript 0 denotes traceless $\Ext$.}}
$$\Ext^1({I}\udot, {I}\udot)_0, \ \
\Ext^2({I}\udot, {I}\udot)_0.$$
The virtual class 
$$[P_n(X,\beta)]^{vir} \in A_{\text{dim}^{vir}}
\left(P_n(X,\beta),\mathbb{Z}\right)$$
is then obtained by standard methods \cite{BehFan,LiTian}.
The virtual dimension is
$$\text{dim}^{vir} = \int_\beta c_1(T_X).$$

Apart from the derived category deformation theory,
the construction of the virtual class of $P_n(X,\beta)$
is parallel to virtual class construction in DT theory \cite{Thomas}.

\subsection{Characterization}
Consider the kernel/cokernel exact sequence associated to a stable
pair $(F,s)$,
\beq \label{IOFQ}
0\to\I_{C_F}\to\OO_X\Rt{s}F\to Q\to0.
\eeq
The kernel is the ideal sheaf of the Cohen-Macaulay support
curve $C_F$ by Lemma 1.6 of \cite{pt}. The cokernel
$Q$ has dimension 0 support by stability.
The {\em reduced} support scheme, $\text{Support}^{red}(Q)$, is 
called the {\em zero locus} of the pair.
The zero locus lies on $C_F$.

Let $C\subset X$ be a fixed Cohen-Macaulay curve.
Stable pairs with support  $C$ and bounded zero locus are characterized
as follows.
Let $$\m\subset\OO_C$$ 
be the ideal in $\OO_C$ of a 0-dimensional subscheme. 
Since $$\hom(\m^r/\m^{r+1},\OO_C)=0$$ by
the purity of $\OO_C$, we obtain an inclusion $$\hom(\m^r,\OO_C)\subset
\hom(\m^{r+1},\OO_C).$$ 
The inclusion $\m^r\into\OO_C$ induces a canonical section 
$$\OO_C\into\hom(\m^r,\OO_C).$$

\begin{prop} \label{descl}
A stable pair $(F,s)$ with support $C$ satisfying
$$\text{\em Support}^{red}(Q) \subset \text{\em Support}(\OO_C/\m)$$ 
is equivalent to a subsheaf
of $\hom(\m^r,\OO_C)/\OO_C,\ r\gg0.$
\end{prop}

Alternatively, we may work with coherent subsheaves of the quasi-coherent sheaf
\begin{equation}\label{infhom}
\lim\limits_{\To}\hom(\m^r,\OO_C)/\OO_C
\end{equation}
Under the equivalence of Proposition \ref{descl}, the 
subsheaf of \eqref{infhom} corresponds to $Q$, giving a subsheaf $F$ of
$\lim\limits_{\To}\hom(\m^r,\OO_C)$ containing the canonical subsheaf $\OO_C$
and the sequence
$$0\to\OO_C\stackrel{s}{\rightarrow} F \rightarrow Q \rightarrow 0.$$
 Proposition \ref{descl} is proven in \cite{pt}.


\section{$\T$-fixed points with one leg} \label{ttt}

\subsection{Affine chart}
Let $N$ be the 3-fold total space of
$$\OO_{\PP^1} \oplus \OO_{\PP^1} \rightarrow \PP^1 \ $$
carrying the action of the 3-dimensional torus $\T$ as in Section 
\ref{legger}.
Let
\begin{equation}\label{vqaa}
[\OO_N \stackrel{s}{\rightarrow} F] \in P_n(N,d)^\T
\end{equation}
be a $\T$-fixed stable pair. The curve class is
$d[\PP^1]$.

Let $U\subset N$ be the $\T$-invariant affine chart
associated to the 
 $\T$-fixed point of $N$ lying over $0\in \PP^1$.
The restriction of the stable pair \eqref{vqaa} to 
the chart $U$,
\begin{equation}\label{vvvt}
\OO_{U} \stackrel{s_U}{\rightarrow} F_U\ ,
\end{equation}
determines an invariant section $s_U$ of an
equivariant sheaf $F_U$.

Let $x_1,x_2,x_3$ be coordinates on the affine chart $U$
in which the $\T$-action takes the diagonal form,
$$(t_1,t_2,t_3) \cdot x_i = t_i x_i.$$
By convention, $x_1$ and $x_2$ are coordinates on the fibers of
$N$ and $x_3$ is a coordinate on the base $\PP^1$.

We will characterize the restricted data $(F_U,s_U)$
in the coordinates $x_i$ closely following the presentation
of \cite{pt2}. 


\subsection{Monomial ideals and partitions}
Let $x_1,x_2$ be coordinates on the plane 
$\C^2$.
A subscheme $S\subset \C^2$ invariant under the 
action of the diagonal torus,
$$(t_1,t_2)\cdot x_i = t_ix_i$$
must be defined by a monomial ideal
$\I_S \subset \C[x_1,x_2]$.
If
$$\dim_\C \C[x_1,x_2]/\I_S < \infty$$
then $\I_S$ determines a finite partition $\mu_S$
by considering lattice points corresponding
to monomials of $\C[x_1,x_2]$
{\em not} contained in $\I_S$. 
Conversely, each partition $\mu$ determines a monomial ideal
$$\mu[x_1,x_2]\subset \C[x_1,x_2].$$

Similarly, the subschemes $S\subset \C^3$ invariant under
the diagonal $\T$-action are in bijective correspondence with
$3$-dimensional partitions.

\subsection{Cohen-Macaulay support}
The first step in the characterization of the restricted
data
\eqref{vvvt}
is to determine the scheme-theoretic support  
$C_U$ of $F_U$. If nonempty, $C_U$ is a
$\T$-invariant, Cohen-Macaulay subscheme of pure dimension 1.

The $\T$-fixed subscheme $C_U \subset \C^3$ is defined by a 
monomial ideal $$\I_C \subset \C[x_1,x_2,x_3].$$
associated to the 3-dimensional partition $\pi$.
The localisation
$$(\I_C)_{x_3} \subset \C[x_1,x_2,x_3]_{x_3},$$
is $T$-fixed and corresponds to a 2-dimensional partition 
$\mu$.
Alternatively, the 2-dimensional partitions $\mu$ can be defined
as the infinite limit of 
the $x_3$-constant cross-sections of $\pi$.
In order for $C_U$ to have dimension 1,
$\mu$ can not be empty.

There exists
a unique {\em minimal} $\T$-fixed subscheme 
$$C_{\mu}\subset \C^3$$
with outgoing partition $\mu$.
The $3$-dimensional partition corresponding to $C_\mu$ is
the infinite cylinder on
the $x_3$-axis determined by the $2$-dimensional
partitions $\mu$.
Let 
\begin{eqnarray*}
\I_{\mu}= \mu[x_1,x_2] \cdot \C[x_1,x_2,x_3],& \ \ & 
C_{\mu}= \OO_{\C^3}/\I_{\mu}\ .
\end{eqnarray*}

\label{cmmm}
\subsection{Module $M_3$}
The kernel/cokernel sequence associated to the
$\T$-fixed restricted data \eqref{vvvt} takes the form
\begin{equation}\label{cvrw}
0 \rightarrow \I_{C_\mu} \rightarrow \OO_{U} \stackrel{s}
{\rightarrow} F_U \rightarrow Q_U \rightarrow 0\ 
\end{equation}
for an outgoing partition $\mu$.

Since the
support of the quotient $Q_U$ in \eqref{cvrw} is 0-dimensional
by stability and $\T$-fixed, 
$Q_U$ 
must be supported at the origin.
By Proposition \ref{descl}, the pair $(F_U,s_U)$
corresponds to a $\T$-invariant subsheaf of
$$\lim\limits_{\To}\hom(\m^r,\OO_{C_\mu})/\OO_{C_\mu}
,$$
where 
$\m$  is the
ideal sheaf of the origin in $C_\mu\subset\C^3$.
Let
$$M_3 = (\OO_{C_{\mu}})_{x_3}$$
be the $\C[x_1,x_2,x_3]$-module  obtained by localisation. Explicitly
$$
M_3=\C[x_3,x_3^{-1}]\otimes\frac{\C[x_1,x_2]}{\mu[x_1,x_2]}\,.
$$
By elementary algebraic arguments,
\begin{eqnarray*}
\lim\limits_{\To}\hom(\m^r,\OO_{C_\mu}) & 
 \cong &  M_3\ .
\end{eqnarray*}
The $\T$-equivariant
$\C[x_1,x_2,x_3]$-module $M_3$ has a canonical $\T$-invariant
element 1. 
By Proposition \ref{descl}, the $\T$-fixed pair $(F_U,s_U)$
corresponds to a finitely generated $\T$-invariant
 $\C[x_1,x_2,x_3]$-submodule
\beq\label{datum}
Q_U\subset M_3/\langle 1 \rangle.
\eeq

Conversely, {\em every} finitely generated{\footnote{Here, finitely generated
is equivalent to finite dimensional or Artinian.}}
$\T$-invariant $\C[x_1,x_2,x_3]$-sub\-module 
$$Q \subset M_3/\langle 1 \rangle$$
occurs as the restriction to $U$ of
a $\T$-fixed stable pair on $N$.

\subsection{The 1-leg stable pairs vertex}

\label{vc}
Let $R$ be the coordinate ring,
$$
R = \C[x_1,x_2,x_3] \cong \Gamma(U).
$$
Following
the conventions of Section \ref{legger}, the $\T$-action on $R$ is 
\begin{equation*}
  (t_1,t_2,t_3)\cdot x_i = t_i x_i  \,.
\end{equation*}
Since the tangent spaces are dual to the coordinate functions,
the tangent weight of $\mathbf{T}$ along the third axis is $-s_3$.

Let $Q_U \subset M/\langle 1 \rangle$ be a $\T$-invariant
submodule viewed as a stable pair on $U$.
Let $\Ib_U\udot$ denote the universal  complex on
$[Q_U] \times U$.
Consider a $\T$-equivariant free
resolution{\footnote{Here, $\Ib_U\udot$ is viewed to
live in degrees 0 and -1.}}
 of $\Ib_U\udot$,
\begin{equation}
  \label{resol}
 \{ \F_{s} \rightarrow \dots \rightarrow \F_{-1}\}
\cong 
\Ib_U\udot \ \in D^b([{Q}_U] \times U).
\end{equation}
Each term in \eqref{resol}
can be taken to have the form
$$
\F_i = \bigoplus_j 
 R(d_{ij})\,, \quad d_{ij} \in \Z^3.$$
The Poincar\'e polynomial
$$
P_U = \sum_{i,j} (-1)^{i+1} \  t^{d_{ij}}
\  
\in \Z[t_1^\pm,t_2^\pm,t_3^\pm]$$
does not depend on the choice of the resolution
\eqref{resol}.

We denote the $\T$-character of  $F_U$ by $\FFF_U$.
By the
sequence
$$0 \rightarrow \OO_{C_U} \rightarrow F_U \rightarrow Q_U 
\rightarrow 0,$$
we have a complete understanding of
the representation $\FFF_U$. 
The $\T$-eigenspaces of $F_U$ correspond to the
$\T$-eigenspaces of $\OO_{C_U}$ and 
 $Q_U$. 
The result determines 
$$\FFF_U \in \Z(t_1,t_2,t_3).$$
The rational dependence on the $t_i$ is elementary.

From
the resolution \eqref{resol}, we see
that the Poincar\'e polynomial
$P_U$ is related to the $\T$-character
of $F_U$ as follows:
\begin{equation}
\FFF_U =
\frac{1+P_U}{(1-t_1)(1-t_2)(1-t_3)} \label{PQ}
\,.
\end{equation}

The virtual represention
$\chi(\Ib_U\udot,\Ib_U\udot)$ is given by the
following alternating sum
\begin{align*}
\chi(\Ib_U\udot,\Ib_U\udot)  &= \sum_{i,j,k,l} (-1)^{i+k}
 \Hom_R(R(d_{ij}), R(d_{kl}))
\\
&=  \sum_{i,j,k,l} (-1)^{i+k}
R(d_{kl}-d_{ij})\,.
\end{align*}
Therefore, the $\T$-character is 
$$
\tr_{\chi(\Ib_U,\Ib_U)} =
\frac{P_U  \,\overline{P}_U}
{(1-t_1)(1-t_2)(1-t_3)} \,.
$$
The bar operation 
$$\gamma \in \Z(\!(t_1,t_2,t_3)\!) \mapsto 
\Z(\!(t_1^{-1},t_2^{-1},
t_3^{-1})\!)$$
is 
$t_i \mapsto t_i^{-1}$
on the variables.

We find the $\T$-character of 
the $U$ summand of  virtual tangent
space $\mathcal{T}_{\left[{I}\udot\right]}$
of the moduli space of stable pairs of the 1-leg cap is
$$ \tr_{R-\chi(\Ib\udot_U,\Ib_U\udot)} =
\frac{1-P_U \, \overline{P}_U}
{(1-t_1)(1-t_2)(1-t_3)} \,  ,
$$
see \cite{pt2}.
Using \eqref{PQ}, we may express the answer in terms of
$\FFF_U$,
\begin{equation}\label{vertexchar}
  \tr_{R-\chi(\Ib\udot_U,\Ib_U\udot)} 
= \FFF_{U} -
\frac{\overline{\FFF}_U}{t_1t_2t_3} +  \FFF_{U}
\overline{\FFF}_U \frac{(1-t_1)(1-t_2)(1-t_3)}{t_1 t_2 t_3} \,.
\end{equation}
On the right side of
 \eqref{vertexchar}, the rational functions
should be expanded 
in ascending powers in the $t_i$.

The stable pairs vertex is obtained from \eqref{vertexchar}
after a redistribution of edge terms following \cite{pt2}.
Let 
$$
\FFF_{\mu}  = \sum_{(k_1,k_2) \in \mu} t_1^{k_1}
t_2^{k_2}\ 
$$
correspond to the outgoing partition $\mu$.
Define
$$
\GGG_{\mu} = - \FFF_{\mu} -
\frac{\overline{\FFF}_{\mu}}{t_1 t_2} +  \FFF_{\mu}
\overline{\FFF}_{\mu} \frac{(1-t_1)(1-t_2)}{t_1 t_2} \,.
$$
Define the vertex character $\bV_U$ by the following
modification,
\begin{equation}\label{gx34}
\bV_U = \tr_{R-\chi(\Ib\udot_U,\Ib_U\udot)} 
 + 
\frac{\GGG_{\mu}(t_{1},t_{2})}{1-t_3}\, . 
\end{equation}
The character $\bV_U$ depends
{\em only on the local data ${Q}_U$}.
By the results of \cite{pt2}, $\bV_U$ is a 
Laurent polynomial in $t_1$, $t_2$, and $t_3$.

\subsection{Descendents}
Let $[0]\in H^*_\T(\PP^1,\Z)$ be the class of the
$\T$-fixed point
$0\in \PP^1$. Consider the $\T$-equivariant
descendent (with value in the $\T$-equivariant cohomology
of a point),
\begin{multline}\label{vpzz}
\left\langle \tau_{i_1}([0]) \cdots \tau_{i_k}([0]) \right
\rangle_{n,d}^N =\\
\int_{P_n(N,d)} \prod_{j=1}^k \tau_{i_j}([0])
\Big( [P_{n}(N,d)]^{vir}\Big)\in \Q(s_1,s_2,s_3)\ ,
\end{multline}
following the notation of  Section \ref{dess}.

In order to calculate \eqref{vpzz} by $\T$-localization, we
must determine the action of
the operators $\tau_{i}([0])$ on the $\T$-equivariant
cohomology of the $\T$-fixed loci.
The calculation of \cite{pt2} yields a formula for the
descendent weight,
\begin{multline}
\bw_{{i_1},\cdots, {i_m}}
(Q_U) =\\
 e(-\bV_{U})
\cdot \prod_{j=1}^m  
\text{ch}_{2+i_j}\big(\FFF_{U}\cdot 
(1-t_1)(1-t_2)(1-t_3)\big) \ .
\end{multline}
The {\em descendent vertex} $\bW_\mu^{\mathsf{Vert}}(\tau_{i_1}([0]) \cdots 
\tau_{i_m}([0]))$
is obtained from the descendent weight,
\begin{multline}\label{vvped}
\bW_\mu^{\mathsf{Vert}} (\tau_{i_1}([0]) \cdots \tau_{i_k}([0]))
 = \\
\left(\frac{1}{s_1s_2}\right)^k
\sum_{Q_U}
 \bw_{{i_1}, \cdots, {i_k}}
(Q_U)\ q^{\ell({Q}_U)+|\mu|}\ \in
\Q(s_1,s_2,s_3)(\!(q)\!)\ .
\end{multline}
\label{heyle}
Here, $\ell(Q_U)$ is the length of $Q_U$.

\subsection{Edge weights}
The edge weight in the cap geometry  is
$$\bW^{(0,0)}_\mu = e(\GGG_\mu)\  \in \Q(s_1,s_2).$$
In fact, $\bW^{(0,0)}_\mu$ is simply the inverse product
of the tangent weights of the Hilbert scheme of points of $\C^2$
at the $T$-fixed point corresponding to the partition $\mu$.


\section{Capped 1-leg descendents: stationary} \label{333}
\subsection{Overview}
Consider the capped geometry of Section \ref{legger}. As before,  let
 $0\in \PP^1$ be the $\mathbf{T}$-fixed point away
from the relative divisor over $\infty \in \PP^1$, and 
let 
$$[0] \in H^*_{\mathbf{T}}(\PP^1, \mathbb{Z})$$
be the associated class.
The $\mathbf{T}$-weight on the tangent space to $\PP^1$ at 0
is $-s_3$.
We study here the stationary{\footnote{Stationary
refers to descendents of point classes.}} series 
\begin{equation}
{\mathsf Z}^{\mathsf{cap}}_{d,\eta} 
\left(   \prod_{j=1}^k \tau_{i_j}([0])
\right)^{\mathbf{T}}\ .
\label{gbn}
\end{equation}
Our main result is a special case of Theorem \ref{cnnn}.
\begin{prop}
\label{cttt} 
$
{\mathsf Z}^{\mathsf{cap}}_{d,\eta} 
\left(   \prod_{j=1}^k \tau_{i_j}([0])
\right)^{\mathbf{T}}$
is the 
Laurent expansion in $q$ of a rational function in 
$\mathbb{Q}(q,s_1,s_2,s_3)$.
\end{prop}

\subsection{Dependence on $s_3$}
The function \eqref{gbn} is the generating series of the
integrals 
\begin{equation} \label{krt}
\left\langle \prod_{j=1}^k \tau_{i_j}([0])
\right\rangle_{\!n,\eta}^{\mathsf{cap},{\mathbf{T}}}
=
\int _{[P_{n} (N/N_\infty,d)]^{vir}}
  \prod_{j=1}^k \tau_{i_j}([0])\ 
\cup \epsilon_\infty^*(C_{\eta})\ ,
\end{equation}
following the notation of Section \ref{legger}.

Let $\ell(\eta)$ denote the length of the partition
$\eta$ of $d$, and let \begin{equation}\label{ktgg}
\delta=\sum_{j=1}^k i_j + d-\ell(\eta)\ .
\end{equation}
The dimension of 
$[P_{n} (N/N_\infty,d)]^{vir}$ after applying
the integrand of \eqref{krt} is $2d-\delta$.

\begin{lem} The 
 integral $\left\langle \prod_{j=1}^k \tau_{i_j}([0])
\right\rangle_{\!n,\eta}^{\mathsf{cap},\mathbf{T}}$ \label{rq2}
is a {\em polynomial} in $s_3$ of degree
$\delta$
with coefficients in the subring
$${\mathbb Q}[s_1,s_2]_{(s_1s_2)}\subset {\mathbb Q}(s_1,s_2).$$
\end{lem}

\begin{proof} 
Let $N=\cO_{\Pp} \oplus \cO_{\Pp}$.
Let
${\mathbb{F}} \rarr {\mathcal N}$
denote the universal sheaf over the universal total space
$${\mathcal N} \rarr P_n(N/N_\infty,d).$$ 
Since $N=\PP^1 \times \com^2$, there is a proper morphism
$${\mathcal N} \rarr P_n(N/N_\infty,d) \times \com^2.$$
The locations and multiplicities of the supports of the universal
sheaf determine a
morphism of Hilbert-Chow type,
$$\iota:P_n(N/N_\infty,d) \rarr \text{Sym}^{d}(\com^2).$$
A $\mathbf{T}$-equivariant, proper
morphism,
$$\widehat{\iota}: \text{Sym}^{d}(\com^2) \rarr \oplus_{1}^{d}
(\com^2),$$
is obtained via the higher moments,
\begin{multline*}
\widehat{\iota}\Big( \ \{(x_i,y_i)\} \ \Big) = \\
\Big(\sum_i x_i, \sum_i y_i\Big) \oplus 
\Big(\sum_i x^2_i, \sum_i y^2_i\Big) \oplus \cdots \oplus 
\Big(\sum_i x^d_i, \sum_i y^d_i\Big).
\end{multline*}
Let $\rho=\widehat{\iota}\circ \iota$.

Since $\rho$ is a $\mathbf{T}$-equivariant, proper morphism,
there is a $\mathbf{T}$-equivariant push-forward
$$\rho_*: A^{\mathbf{T}}_*(P_n(N/N_\infty,d), {\mathbb Q}) \rarr
A^{\mathbf{T}}_*(   \oplus_1^{d}(\com^2)    , {\mathbb Q}).$$
Descendent invariants
are defined via the $\mathbf{T}$-equivariant residue of
$$\left(\prod_{j=1}^k \tau_{i_j}([0]) \cup \epsilon^*_{\infty}(C_\eta)
\right) \ \cap
[P_n(N/S,d)]^{vir} \ \in A^{\mathbf{T}}_*(P_n(N/N_\infty,d), {\mathbb Q}).$$
We may instead calculate the $\mathbf{T}$-equivariant residue of
\begin{equation}\label{ress}
\rho_*\left( \left( 
\prod_{j=1}^k \tau_{i_j}([0]) \cup \epsilon^*_{\infty}(C_\eta)\right) \ \cap
[P_n(N/N_\infty,d)]^{vir}\right)
\end{equation} in
$A^{\mathbf{T}}_*( 
\oplus_1^{d}(\com^2)    , {\mathbb Q})$.

The codimension of the class \eqref{ress} in
$\oplus_1^{d}(\com^2)$
 is $\delta$.
Since the third factor of $\mathbf{T}$ acts trivially on
$\oplus_1^{d}(\com^2)$, the class \eqref{ress} may be written as
\begin{equation}\label{htty3}
\gamma_0 s_3^0 + \gamma_1 s_3^1 + \ldots + \gamma_{\delta} s_3^{\delta}
\end{equation}
where $\gamma_i \in A^{T}_{2d-\delta+i}( 
\oplus_1^{d}(\com^2)    , {\mathbb Q})$.
Since the space $\oplus_1^{d} 
(\com^2)$ has a unique $T$-fixed point
with tangent weights,
$$-s_1,-s_2,-2s_1,-2s_2, \ldots, -ds_1, -ds_2,$$
we conclude the localization of $\gamma_i$
has only monomial poles in the variables $t_1$ and $t_2$.
\end{proof}

As a consequence of Lemma \ref{rq2}, we may write
\begin{equation} \label{krtt}
{\mathsf Z}^{\mathsf{cap}}_{d,\eta} 
\left(   \prod_{j=1}^k \tau_{i_j}([0])
\right)^{\mathbf{T}}
=\sum_{r=0}^\delta  s_3^r \cdot \Gamma_r(q,s_1,s_2)
\end{equation}
where $\Gamma_r \in \mathbb{Q}(s_1,s_2)((q))$.

\subsection{Localization: rubber contribution} \label{rubcon}
The $\mathbf{T}$-equivariant
localization formula for the series 
${\mathsf Z}^{\mathsf{cap}}_{d,\eta} 
\left(   \prod_{j=1}^k \tau_{i_j}([0])
\right)^{\mathbf{T}}$
has three parts:
\begin{enumerate}
\item[(i)] vertex terms over $0\in \PP^1$,
\item[(ii)] edge terms,
\item[(iii)] rubber integrals over $\infty \in \PP^1$.
\end{enumerate}
The vertex and edge terms have been explained already in Section \ref{ttt}.
We discuss the rubber integrals here.

The stable pairs theory of {\em rubber}{\footnote{We
follow the terminology and conventions of the
parallel rubber discussion for the local Donaldson-Thomas
theory of curves treated in \cite{lcdt}.}} naturally arises at the
boundary of $P_n(N/N_\infty,d)$.
Let $R$ be a rank 2 bundle of level $(0,0)$ over $\Pp$. Let 
 $$R_0, R_\infty\subset R$$ 
denote the fibers over $0, \infty\in \Pp$.
The 1-dimensional torus $\C^*$ acts on $R$ via the symmetries of
$\Pp$. 
Let $P_n(R/R_0\cup R_\infty,d)$ be the relative moduli space
of stable pairs, and let
$$P_n(R/R_0 \cup R_\infty,d)^\circ \subset P_n(R/R_0\cup R_\infty,d)$$
denote the open set with finite stabilizers for the $\C^*$-action
and {\em no} destabilization over $\infty\in \Pp$.
The rubber moduli space,
$${P_n(R/R_0\cup R_\infty,d)}^\sim  
= P_n(R/R_0 \cup R_\infty,d)^\circ/\C^*,$$
denoted by a superscripted tilde,
is determined by the (stack) quotient. The moduli space is 
empty unless $n>d$.
The rubber theory of $R$ is defined by integration against the
rubber virtual class,
 $$[{P_n(R/R_0\cup R_\infty,d)}^\sim ]^{vir}.$$ 
All of the above rubber constructions are $T$-equivariant for the
scaling action on the fibers of $R$ with weights $s_1$ and $s_2$.

The rubber moduli space $P_n(R/R_0\cup R_\infty, d)^\sim$ carries
a cotangent line at the dynamical point $0 \in \Pp$. Let
$$\psi_0 \in A^1_T({P_n(R/R_0\cup R_\infty,d)}^\sim, {\mathbb Q})$$
denote the associated cotangent line class.
Let $$\mathsf{P}_\mu \in A^{2d}_T(\text{Hilb}(\C^2,d),\mathbb{Z})$$
be the class corresponding to the $T$-fixed point determined
by the monomial ideal $\mu[x_1,x_2]\subset \C[x_1,x_2]$.

In the localization formula for the cap, special
rubber integrals with relative conditions $\mathsf{P}_\mu$ over $0$ and $\CC_\eta$
(in the Nakajima basis) over $\infty$ arise. Let
\begin{equation*} 
\mathsf{S}^\mu_\eta =   
\sum_{n\geq d} q^{n}
\left\langle \mathsf{P}_\mu \ \left| \ \frac{1}{s_3-\psi_0}  \ \right|\ \CC_\eta 
\right\rangle_{n,d}^{
\sim}\ \in \Q(s_1,s_2,s_3)((q)) \ .
\end{equation*}
The bracket on the right is the rubber
integral defined by $T$-equivariant
residues. If $n=d$, the rubber moduli space in undefined ---
the bracket is then taken to be the $T$-equivariant intersection pairing
between the classes $\mathsf{P}_\mu$ and $\CC_\eta$ in 
$\text{Hilb}(\C^2,d)$.

The $s_3$ dependence of the rubber integral
$$ \left\langle \mathsf{P}_\mu \ \left| \ \frac{1}{s_3-\psi_0}  \ \right|\ \CC_\eta 
\right\rangle_{n,d}^{
\sim}\ \in \Q(s_1,s_2,s_3)$$
enter {\em only} through the term $s_3-\psi_0$.
On the $T$-fixed loci of the moduli space $P_n(R/R_0\cup R_\infty, d)^\sim$, the
cotangent line class $\psi_0$ is either equal to 
a weight of $\text{Tan}_\mu$  (if $0$ lies on a twistor component)
or is nilpotent (if $0$ lies on a non-twistor component).
We conclude the following result.

\begin{lem} The evaluation of $\mathsf{S}_\eta^\mu$ at \label{plle}
$$s_3= n_1 s_1 + n_2 s_2,\ \ \ \ n_1,n_2\in \Q$$
is well-defined if $(n_1,n_2) \neq (0,0)$
and $n_1 s_1 + n_2 s_2$ is not a weight of $\text{\em Tan}_\mu$.
\end{lem}

The weights of $\text{Tan}_\mu$ are either proportional to
$s_1$ or $s_2$ or of the form
$$n_1s_1+n_2s_2 ,\ \ \ \ n_1,n_2\neq 0$$
where $n_1$ is the {\em opposite} sign of $n_2$.

\subsection{Localization: full formula}

The localization formula \cite{GraberP} for the capped 1-leg descendent
vertex is the following:
\begin{equation}\label{fred}
{\mathsf Z}^{\mathsf{cap}}_{d,\eta} 
\left(   \prod_{j=1}^k \tau_{i_j}([0])
\right)^{\mathbf{T}} = \sum_{|\mu|=d}
\bW_\mu^{\mathsf{Vert}} \left(\prod_{j=1}^k \tau_{i_j}([0])      \right) \cdot
{\bW_\mu^{(0,0)}} \cdot \mathsf{S}^{\mu}_{\eta}\ .
\end{equation}
The form is the same as the Donaldson-Thomas localization formulas
used in \cite{moop,lcdt}.

\subsection{Proof of Proposition \ref{cttt}} \label{ggtt2}
We will consider the evaluations of ${\mathsf Z}^{\mathsf{cap}}_{d,\eta} 
(   \prod_{j=1}^k \tau_{i_j}([0]))^{\mathbf{T}}$ at the values
\begin{equation}\label{gthh4}
s_3 = \frac{1}{a}(s_1+s_2)
\end{equation}
for all integers $a>0$.
By Theorem \ref{canpole}, the main cancellation of poles result
of Section \ref{polecan}, the evaluation \eqref{gthh4} of
$\bW_\mu^{\mathsf{Vert}} \left(\prod_{j=1}^k \tau_{i_j}([0])      \right)$
is well-defined and yields a Laurent {\em polynomial} in $q$
with coefficients in $\Q(s_1,s_2)$.
The edge term $\bW_\mu^{(0,0)}$ has no $s_3$ dependence (and
$q$ dependence given by $q^{-d}$).
The evaluation \eqref{gthh4} of
$\mathsf{S}^\mu_\eta$ is well-defined by Lemma \ref{plle} and 
is the Laurent series
associated to a rational function in $\Q(q,s_1,s_2)$
by Lemma \ref{hyy3} below.

We have proven the evalution of ${\mathsf Z}^{\mathsf{cap}}_{d,\eta} 
\left(   \prod_{j=1}^k \tau_{i_j}([0])
\right)^{\mathbf{T}}$
at \eqref{gthh4}  for all integers $a>0$ is well-defined and yields
a rational function in $\Q(q,s_1,s_2)$. 
By \eqref{krtt}  and the invertibility of the
Vandermonde matrix, we see 
$$\Gamma_r(q,s_1,s_2) \in \Q(q,s_1,s_2)$$
for all $0 \leq r \leq \delta$. 
\qed

\subsection{Evaluation of $\mathsf{S}_\eta^\mu$}
\label{dfv}
The following result is well-known from the study of the
quantum differential equation of the Hilbert scheme of 
points \cite{hilb1,hilb2}. We include the proof
for the reader's convenience.

\begin{lem} For all integers $a\neq 0$, the evaluation
$$\mathsf{S}_\eta^\mu
|_{s_3=\frac{1}{a}(s_1+s_2)}$$ 
yields the Laurent series associated to a rational
function in $\Q(q,s_1,s_2)$. \label{hyy3}
\end{lem}

\begin{proof}
Let $\com^*$ act on $\PP^1$ with tangent weights $-s_3$ and $s_3$
at $0,\infty \in \PP^1$ respectively.
Lift the $\com^*$-action to $\cO_{\Pp}(-a)$ 
with fiber weights{\footnote{Remember, weights on the
coordinate functions are the opposite of the weights on the fibers.}} $as_3$ and $0$ over $0,\infty\in \PP^1$.
Lift $\com^*$ to $\cO_{\Pp}$ with fiber
weights $0$ and $0$ over $0,\infty\in \PP^1$.
The $(-a,0)$-tube is the geometry of total
space of
\begin{equation} \label{gttr}
\cO_{\Pp}(-a) \oplus \cO_{\Pp} \rightarrow \Pp
\end{equation}
relative to the fibers over both $0,\infty \in \Pp$.

The 2-dimensional torus $T$ acts on the $(-a,0)$-tube
as before by scaling the line summands. For
$$\mathbf{T}=T \times \com^* ,$$
we obtain a $\mathbf{T}$-action on the $(-a,0)$-tube.
Define the generating series of $\mathbf{T}$-equivariant integrals
\begin{equation} \label{hllw} 
{\mathsf Z}^{(-a,0),\mathbf{T}}_{d,\eta^0,\eta^\infty} 
= \sum_{n} q^{n}
\Big\langle \CC_{\eta^0} \ \Big| \ 1 \ \Big|\ \CC_{\eta^\infty} 
\Big\rangle_{n,d}^{(-a,0)}\ \in \Q(s_1,s_2,s_3)((q))\ 
\end{equation}
where the superscript $(-a,0)$ refers to the geometry \eqref{gttr}.

The series ${\mathsf Z}^{(-a,0),\mathbf{T}}_{d,\eta^0,\eta^\infty} $ 
has no insertions. Hence, the results of 
\cite{moop,mpt} show 
${\mathsf Z}^{(-a,0),\mathbf{T}}_{d,\eta^0,\eta^\infty}$
is actually the Laurent  series associated to a rational function
in $\Q(q,s_1,s_2,s_3)$.
The $\mathbf{T}$-equivariant localization formula yields
$${\mathsf Z}^{(-a,0),\mathbf{T}}_{d,\eta^0,\eta^\infty} =
\sum_{|\mu|=d} \mathsf{S}_{\eta^0}^\mu\Big|_{s_1=s_1-as_3,s_2,s_3=-s_3} \cdot \bW^{(-a,0)}_\mu \cdot \mathsf{S}_{\eta^\infty}^\mu
\ .$$
The formula for the edge term  $\bW^{(-a,0)}_\mu$ can be found in Section 4.6
of \cite{pt2}.

Next, we consider the evaluation of the three terms of the
localization formula at
\begin{equation}\label{jjttf}
s_3= \frac{1}{a}({s_1+s_2})\ .
\end{equation}
After evaluation, the first term becomes
\begin{equation} \label{oldd}
\mathsf{S}_{\eta^0}^\mu\Big|_{s_1=-s_2,s_2,s_3=-s_3}
\end{equation}
which only has $q^d$ terms by holomorphic symplectic vanishing \cite{mpt,lcdt}.
The evaluation of $\bW^{(-a,0)}_\mu$ at \eqref{jjttf} is easily
seen to be well-defined and nonzero by inspection of the formulas in
Section 4.6 of \cite{pt2}.
The $q$ dependence of $\bW^{(-a,0)}_\mu$ is monomial.
The evaluation of the third term
$ \mathsf{S}_{\eta^\infty}^\mu\ $
at \eqref{jjttf} is well-defined by  Lemma \ref{plle}.
We conclude the evaluation of ${\mathsf Z}^{(-a,0),\mathbf{T}}_{d,\eta^0,\eta^\infty} $
at \eqref{jjttf} is a well-defined rational function in
$\Q(q,s_1,s_2)$.

By the invertibility of \eqref{oldd} and the edge terms, 
$\mathsf{S}_{\eta^\infty}^\mu$ must also be a rational function in
$\Q(q,s_1,s_2)$ after the evaluation \eqref{jjttf}.
\end{proof}


\subsection{Twisted cap}
The twisted $(a_1,a_2)$-cap is the geometry of the total space of
\begin{equation} \label{gttrr}
\cO_{\Pp}(a_1) \oplus \cO_{\Pp}(a_2) \rightarrow \Pp
\end{equation}
relative to the fiber over  $\infty \in \Pp$.

We lift the $\com^*$-action on $\Pp$ to $\cO_{\Pp}(a_i)$ 
with fiber weights $0$ and $-a_is_3$ over $0,\infty\in \PP^1$.
The 2-dimensional torus $T$ acts on the $(a_1,a_2)$-cap
by scaling the line summands, so
we obtain a $\mathbf{T}$-action on the $(a_1,a_2)$-cap.
Define the generating series of $\mathbf{T}$-equivariant integrals
\begin{multline*}  
{\mathsf Z}^{(a_1,a_2)}_{d,\eta} \left(   \prod_{j=1}^k \tau_{i_j}([0])
\right)^{\mathbf{T}}
= \\ 
\sum_{n} q^{n}
\left\langle \prod_{j=1}^k \tau_{i_j}([0]) \   \Bigg|\ \CC_{\eta} 
\right\rangle_{n,d}^{(a_1,a_2)}\  \in \Q(s_1,s_2,s_3)((q))\ 
\end{multline*}
where the superscript $(a_1,a_2)$ refers to the geometry \eqref{gttrr}.

\begin{prop}
\label{ctttt} 
$
{\mathsf Z}^{(a_1,a_2)}_{d,\eta} 
\left(   \prod_{j=1}^k \tau_{i_j}([0])
\right)^{\mathbf{T}}$
is the 
Laurent expansion in $q$ of a rational function in 
$\mathbb{Q}(q,s_1,s_2,s_3)$.
\end{prop}

\begin{proof}
The twisted $(a_1,a_2)$-cap admits a $\mathbf{T}$-equivariant
degeneration to a standard $(0,0)$-cap and an $(a_1,a_2)$-tube by
bubbling off $0\in \Pp$.
The insertions $\tau_{i_j}([0])$ are sent $\mathbf{T}$-equivariantly
to the non-relative point of the $(0,0)$-cap.
The rationality of $
{\mathsf Z}^{(a_1,a_2)}_{d,\eta} 
\left(   \prod_{j=1}^k \tau_{i_j}([0])
\right)^{\mathbf{T}}$
 then follows from Proposition \ref{cttt}, the 
$\mathbf{T}$-equivariant
rationality
results for the $(a_1,a_2)$-tube without insertions \cite{mpt,lcdt}, and
the degeneration formula.
\end{proof}

\section{Cancellation of poles} \label{polecan}
\subsection{Overview}
Our goal here is to prove the following result.
\begin{thm}\label{canpole}
For all integers $a>0$, the evaluation
\begin{equation*}
\bW^{\Ver}_\mu\left(\prod_{j=1}^k\tau_{i_j}([0])\right)\bigg|_{s_3=\frac{1}{a}(s_1+s_2)}
\end{equation*}
is well-defined and yields a
 Laurent polynomial in $q$ with coefficients in $\Q(s_1,s_2)$.
\end{thm}
We regard the partition $\mu$, the 
descendent factor $\prod_{j=1}^k\tau_{i_j}([0])$, and 
the integer $a$ as fixed throughout Section \ref{polecan}.

Recall  $\bW^{\Ver}_\mu\left(\prod_{j=1}^k\tau_{i_j}([0])\right)$ is 
defined as an infinite sum over the fixed loci $Q_U$,
\begin{equation}\label{infinite sum}
\bW^{\Ver}_\mu\left(\prod_{j=1}^k\tau_{i_j}([0])\right) = \left(\frac{1}{s_1s_2}\right)^k\sum_{Q_U}\bw_{\tau_{i_1},\ldots,\tau_{i_k}}(Q_U) q^{l(Q_U)+|\mu|}.
\end{equation} 
The $Q_U$ are determined by $\FFF_U$, the weight of the corresponding box configuration. Although $\FFF_U$ is just a Laurent series in $t_1,t_2,t_3$, the
product
$(1-t_3)\FFF_U$ 
is a Laurent polynomial. 

Our approach to proving Theorem~\ref{canpole} 
is to break (\ref{infinite sum}) into finite sums 
based on the Laurent polynomial 
$$(1-t_3)\FFF_U|_{t_3=(t_1t_2)^{\frac{1}{a}}}\ .$$ 
For any Laurent polynomial $f\in\Z[t_1,t_2,(t_1t_2)^{-\frac{1}{a}}]$, define 
\begin{equation*}
\SSS_f = \left\{Q_U \ \bigg| \ (1-t_3)\FFF_U|_{t_3=(t_1t_2)^{\frac{1}{a}}} = f
\right\} \ .
 \end{equation*}
Theorem~\ref{canpole} follows from the following result
 regarding the subsums of (\ref{infinite sum}) 
corresponding to the sets $\SSS_f$.
\begin{prop}\label{vanishing}
Let $f\in\Z[t_1,t_2,(t_1t_2)^{-\frac{1}{a}}]$ be a Laurent polynomial. The evaluation
\begin{equation*}
\left(\sum_{Q_U\in\SSS_f}\bw_{i_1,\ldots,i_k}(Q_U)\right)\bigg|_{s_3=\frac{1}{a}(s_1+s_2)}
\end{equation*}
is well-defined. Moreover, the evaluation vanishes for all but finitely many choices of $f$.
\end{prop}

\subsection{Notation and Preliminaries}

We introduce here the notation and conventions required
 to analyze the sums appearing in Proposition~\ref{vanishing}.

First, we view the partition $\mu$ as a subset of $\Z_{\ge 0}^2$.
The lattice points, 
 for which we use the coordinates $(i,j)\in\mu$,
correspond to the lower left corners of
the boxes of $\mu$.
  We also write 
$$(\delta; j) = (i,j)$$ for $\delta = i-j$.

The points $(\delta; j)\in\mu$ for fixed $\delta$ lie on a single diagonal. 
The diagonals will play an important role. Let 
 $\mu_\delta = \{j \mid (\delta; j)\in\mu\}$, and  
define
\begin{equation*}
\Sym_\mu = \prod_{\delta\in\Z}\Sym(\mu_\delta),
\end{equation*}
where $\Sym(S)$ is the group of permutations of a set $S$. Thus,
 $\Sym_\mu$ may be viewed 
as the group of permutations of $\mu$ which move points only 
inside their diagonals. 
Let $$\sgn:\Sym_\mu\to\{\pm 1\}$$ be the sign of the permutation of $\mu$.

Recall the Laurent polynomials $(1-t_3)\FFF_U$ are of the form
\begin{equation*}
(1-t_3)\FFF_U = \sum_{(i,j)\in\mu}t_1^it_2^jt_3^{-h_U(i,j)},
\end{equation*}
where $h_U(i,j)$ is the depth of the box arrangement below $(i,j)$.

Because of our reparametrization of the partition $\mu$ and the evaluation $t_3=(t_1t_2)^{\frac{1}{a}}$,  the
following change of variables will be convenient: 
\begin{equation*}
v_1=t_1, \quad v_2=t_1t_2, \quad v_3=t_1t_2t_3^{-a}
\end{equation*}
and $u_i=e(v_i)$, so
\begin{equation*}
u_1 = s_1, \quad u_2=s_1+s_2, \quad u_3=s_1+s_2-as_3.
\end{equation*}
The evaluations under consideration are then simply $v_3=1$ and $u_3=0$. 

 From now on we will assume $\SSS_f$ to be  nonempty, so 
$$f = (1-t_3)\FFF_U|_{t_3=(t_1t_2)^{\frac{1}{a}}}$$ 
for some $Q_U$ and thus $f$ can be written in the form
\begin{equation*}
f = \sum_{(\delta; j)\in\mu}v_1^{\delta}v_2^{e_\delta(j)}
\end{equation*}
for some exponents $e_\delta(j)$. These exponents are made unique by requiring that $e_\delta(j)$ is a weakly decreasing function of $j$, for each $\delta$. We generally regard $f$ as fixed and thus do not indicate the $f$-dependence in $e_\delta(j)$.

We now classify all $Q_U\in\SSS_f$. Given any $\sigma=(\sigma_\delta)\in\Sym_\mu$, we define a function $h_\sigma: \mu \to \Z$ by
\begin{equation*}
h_\sigma(\delta; j) = a\cdot(j - e_\delta(\sigma_\delta^{-1}(j))).
\end{equation*}
When $h_\sigma$  defines a valid box arrangement, we say 
 $\sigma$ is {\it admissible}. Admissibility
 is equivalent to the following conditions on $\sigma$:
\begin{align*}
\sigma_0(j)&\ne 0 \text{ if } e_0(j)>0 \\
\sigma_{\delta+1}(j)&\ne\sigma_{\delta}(k) \text{ if } e_{\delta+1}(j)>e_{\delta}(k) \\
\sigma_{\delta}(j)&\ne\sigma_{\delta+1}(k)+1 \text{ if } e_{\delta}(j)>e_{\delta+1}(k)+1.
\end{align*}
For admissible $\sigma$, let $Q_\sigma$ denote the corresponding $\T$-fixed locus.

Unraveling the definitions, we compute
\begin{align*}
(1-t_3)\FFF_\sigma &= \sum_{(i,j) \in \mu}t_1^it_2^jt_3^{-h_\sigma(i-j,j)} \\
&= \sum_{(\delta; j)\in \mu}v_1^\delta v_2^j v_2^{-\frac{1}{a}h_\sigma(\delta; j)}v_3^{\frac{1}{a}h_\sigma(\delta; j)} \\
&= \sum_{(\delta; j)\in \mu}v_1^\delta v_2^{e_\delta(\sigma_\delta^{-1}(j))} v_3^{j - e_\delta(\sigma_\delta^{-1}(j))} \\
&= \sum_{(\delta; j)\in \mu}v_1^\delta v_2^{e_\delta(j)} v_3^{\sigma_\delta(j) - e_\delta(j)}.
\end{align*}
We conclude  $(1-t_3)\FFF_\sigma|_{v_3=1} = f$ and $Q_\sigma\in\SSS_f$. In fact, 
a direct examination shows
every $Q_{U'}\in\SSS_f$ can be obtained as $Q_\sigma$
for some admissible
$\sigma\in \Sym_\mu$.
If we let $\Sym_\mu^0$ be the subgroup of $\Sym_\mu$ consisting of elements $\tau$ such that $e_\delta(\tau_\delta(j))=e_\delta(j)$, then $Q_\sigma = Q_{\sigma'}$ if and only if $\sigma^{-1}\sigma'\in \Sym_\mu^0$.

We thus can replace the sum over $Q_U\in\SSS_f$ with a sum over admissible $\sigma\in\Sym_\mu$:
\begin{multline}\label{kk449}
\left(\sum_{Q_U\in\SSS_f}\bw_{i_1,\ldots,i_k}(Q_U)\right)\bigg|_{s_3=\frac{1}{a}(s_1+s_2)} = \\
 \frac{1}{|\Sym_\mu^0|}\left(\sum_{\sigma\in\Sym_\mu\text{ admissible}}\bw_{i_1,\ldots,i_k}(Q_\sigma)\right)\bigg|_{s_3=\frac{1}{a}(s_1+s_2)}.
\end{multline}
We will show the evaluation is well-defined 
by choosing $\kappa_0$ such that each 
term $\bw_{i_1,\ldots,i_k}(Q_\sigma)$ in the above sum has order of vanishing along $u_3=0$ at least $-\kappa_0$, and then showing 
\begin{equation}\label{differentiated}
\sum_{\sigma\in\Sym_\mu\text{ admissible}}\left(\frac{\partial}{\partial u_3}\right)^\kappa(u_3^{\kappa_0}\bw_{i_1,\ldots,i_k}(Q_\sigma))\bigg|_{u_3=0} = 0
\end{equation}
for $0\le\kappa<\kappa_0$.

The second part of Proposition \ref{vanishing},
the vanishing of the evaluation \eqref{kk449}
 for all but finitely many $f$, is
 then equivalent to proving that (\ref{differentiated}) holds for $\kappa=\kappa_0$ (for all but finitely many $f$).

In order to prove these vanishing results, we will need to analyze the dependence of the terms $\left(\frac{\partial}{\partial u_3}\right)^\kappa(u_3^{\kappa_0}\bw_{i_1,\ldots,i_k}(Q_\sigma))\bigg|_{u_3=0}$ on the permutation $\sigma\in\Sym_\mu$. 
For each $\kappa$, we will find
 the corresponding term is equal to a polynomial in the values $\sigma_\delta(j)$ of relatively low degree which 
 vanishes at all inadmissible permutations $\sigma$.

Let  $\Q[\sigma]$ and $\Q(\sigma)$ denote the
ring of polynomials and the field of rational functions respectively 
in the variables $\sigma_\delta(j)$. For a polynomial $P\in\Q[\sigma]$, let $\deg(P)$ be the (total) degree of $P$. For rational functions $\frac{P}{Q}\in\Q(\sigma)$, we set 
$$\deg\left(\frac{P}{Q}\right) = \deg(P)-\deg(Q).$$

We observe that if $P\in\Q[\sigma]$ has degree $\deg(P)<\sum_{\delta}\frac{1}{2}|\mu_\delta|(|\mu_\delta|-1)$, then
\begin{equation*}
\sum_{\sigma\in\Sym_\mu}\sgn(\sigma)P(\sigma) = 0,
\end{equation*}
since a nonzero alternating polynomial with respect to $\Sym_\mu$ would have to have greater degree.

\subsection{Proof of Proposition~\ref{vanishing}}

We need to study the $\sigma$-dependence of 
\begin{equation*}
\bw_{i_1,\ldots,i_k}(Q_\sigma) = e(-\bV_\sigma)\prod_{j=1}^k\ch_{2+i_j}(\FFF_\sigma\cdot(1-t_1)(1-t_2)(1-t_3)).
\end{equation*}

We begin by explicitly writing $\bV_\sigma$ in terms of $\sigma$ and the numbers $e_\delta(j)$. Recall 
\begin{equation*}
\bV_\sigma = \frac{\FFF'_\sigma -\FFF'_0}{1-t_3} 
+ \frac{\overline{\FFF'_\sigma} -\overline{\FFF'_0}}{t_1t_2(1-t_3)} 
- \frac{\FFF'_\sigma\overline{\FFF'_\sigma} - \FFF'_0\overline{\FFF'_0}}{1-t_3}(1-t_1^{-1})(1-t_2^{-1}),
\end{equation*}
where $\FFF'_\sigma = (1-t_3)\FFF_\sigma$ and $$\FFF'_0 = 
\sum_{(i,j)\in\mu}t_1^it_2^j.$$ 
In particular, $\bV_\sigma|_{v_3=1}$ does not 
depend on $\sigma$. Hence, 
the order of vanishing of $e(-\bV_\sigma)$ along $u_3=0$ is an
 integer $-\kappa_0$ independent of $\sigma$. 
Since the descendent factor is a polynomial in $u_1,u_2,u_3$, 
the order of vanishing of $\bw_{i_1,\ldots,i_k}(Q_\sigma)$ along $u_3=0$ 
is at least $-\kappa_0$. If $\kappa_0\le 0$, 
then the evaluation is well-defined on each $\bw_{i_1,\ldots,i_k}(Q_\sigma)$ 
and thus on their sum. If $\kappa_0<0$, then 
the evaluation in fact yields zero. 
So we may assume $\kappa_0\ge 0$.

We now rewrite $\bV_\sigma$ in terms of $v_1,v_2,v_3$. We find
$\bV_\sigma$ equals
\footnotesize
\begin{align*}
&\ \ \ \sum_{(\delta; j)\in\mu}\frac{v_1^\delta v_2^{e_\delta(j)}v_3^{\sigma_\delta(j) - e_\delta(j)}-v_1^\delta v_2^j}{1-(\frac{v_2}{v_3})^{\frac{1}{a}}}\\
& + \sum_{(\delta; j)\in\mu}\frac{v_1^{-\delta} v_2^{-e_\delta(j)-1}v_3^{-\sigma_\delta(j) + e_\delta(j)}-v_1^{-\delta} v_2^{-j-1}}{1-(\frac{v_2}{v_3})^{\frac{1}{a}}} \\
&- \sum_{(\delta_1; j_1),(\delta_2; j_2)\in\mu}\frac{v_1^{\delta_1-\delta_2} v_2^{e_{\delta_1}(j_1)-e_{\delta_2}(j_2)}v_3^{\sigma_{\delta_1}(j_1)-\sigma_{\delta_2}(j_2)-e_{\delta_1}(j_1)+e_{\delta_2}(j_2)}-v_1^{\delta_1-\delta_2} v_2^{j_1-j_2}}{(1-(\frac{v_2}{v_3})^{\frac{1}{a}}) \cdot(1-v_1^{-1})^{-1}(1-v_1v_2^{-1})^{-1}}.
\end{align*}
\normalsize
Let $C > 2\max(e_\delta(j))$ be a large positive integer. We break up each of the three above sums above using $C$. Then, $\bV_\sigma$ equals
\footnotesize
\begin{align*}
&\ \ \ 
\sum_{(\delta; j)\in\mu}\frac{v_1^\delta v_2^{e_\delta(j)}v_3^{\sigma_\delta(j) - e_\delta(j)}-v_1^\delta v_2^{-C}v_3^{\sigma_\delta(j)+C}}{1-(\frac{v_2}{v_3})^{\frac{1}{a}}} 
\\ &+ \sum_{(\delta; j)\in\mu}\frac{v_1^\delta v_2^{-C}v_3^{j+C}-v_1^\delta v_2^j}{1-(\frac{v_2}{v_3})^{\frac{1}{a}}} \\
&+ \sum_{(\delta; j)\in\mu}\frac{v_1^{-\delta} v_2^{-e_\delta(j)-1}v_3^{-\sigma_\delta(j) + e_\delta(j)}-v_1^{-\delta} v_2^{-C-1}v_3^{-\sigma_\delta(j)+C}}{1-(\frac{v_2}{v_3})^{\frac{1}{a}}} \\ &
+ \sum_{(\delta; j)\in\mu}\frac{v_1^{-\delta} v_2^{-C-1}v_3^{-j+C}-v_1^{-\delta} v_2^{-j-1}}{1-(\frac{v_2}{v_3})^{\frac{1}{a}}} \\
&- \sum_{(\delta_1; j_1),(\delta_2; j_2)\in\mu} \Bigg(
\frac{v_1^{\delta_1-\delta_2} v_2^{e_{\delta_1}(j_1)-e_{\delta_2}(j_2)}v_3^{\sigma_{\delta_1}(j_1)-\sigma_{\delta_2}(j_2)-e_{\delta_1}(j_1)+e_{\delta_2}(j_2)}}
{1-(\frac{v_2}{v_3})^{\frac{1}{a}}}
\\&
\ \ \ \ \ \ \ \ \ \ \ \ \ \ \ \ \ \ \ \ \ \ \ \ \ \ \
 -\frac{v_1^{\delta_1-\delta_2} v_2^{-C}v_3^{\sigma_{\delta_1}(j_1)-\sigma_{\delta_2}(j_2)+C}}{1-(\frac{v_2}{v_3})^{\frac{1}{a}}}\Bigg)
 \cdot(1-v_1^{-1})(1-v_1v_2^{-1}) \\
&- \sum_{(\delta_1; j_1),(\delta_2; j_2)\in\mu}\frac{v_1^{\delta_1-\delta_2} v_2^{-C}v_3^{j_1-j_2+C}-v_1^{\delta_1-\delta_2} v_2^{j_1-j_2}}{1-(\frac{v_2}{v_3})^{\frac{1}{a}}}\cdot(1-v_1^{-1})(1-v_1v_2^{-1}).
\end{align*}
\normalsize

We now expand out the above sums into monomials:
 all of the resulting terms will be of the form 
$$\pm v_1^x v_2^y v_3^{z(\sigma)},$$
 where $x$ and $y$ have no dependence on the permutation $\sigma = (\sigma_\delta)$ and $z\in\Q[\sigma]$ is a linear function of the values $\sigma_\delta(j)$. 
After separating out the monomials with $x=y=0$,  we write
\begin{equation*}
\bV_\sigma = \sum_{(c,0,0,z)\in S}c v_3^{z(\sigma)}+\sum_{\substack{(c,x,y,z)\in S \\ (x,y)\ne(0,0)}}c v_1^x v_2^y v_3^{z(\sigma)},
\end{equation*}
where $S$ is a finite set containing the data of the monomials which
appear
(with 
coefficients $c \in \mathbb{Z}$).
Then 
\begin{equation*}
e\left(-\sum_{(c,0,0,z)\in S}c v_3^{z(\sigma)}\right) = \phi(\sigma)u_3^{-\kappa_0},
\end{equation*} 
for a rational function $\phi=\phi_f\in\Q(\sigma)$ 
which will be explicitly described below.

We analyze first
 the descendent factors in $\bw_{i_1,\ldots,i_k}(Q_\sigma)$.
 The descendent terms  can be expressed in the form
\begin{multline*}
\prod_{j=1}^k\ch_{2+i_j}\left(\FFF_\sigma\cdot(1-t_1)(1-t_2)(1-t_3)\right) 
=\\
 \prod_{j=1}^k\sum_{(c',x,y,z)\in S'_j}c'(xu_1+yu_2+z(\sigma)u_3)^{2+i_j},
\end{multline*}
where the $S'_j$ are more fixed finite sets containing the data of the 
terms which appear. As before,  $z\in\Q[\sigma]$ is linear.
We then find 
\begin{multline*}
u_3^{\kappa_0}\bw_{i_1,\ldots,i_k}(Q_\sigma) = \phi(\sigma)\prod_{\substack{(c,x,y,z)\in S \\
 (x,y)\ne(0,0)}}(xu_1+yu_2+z(\sigma)u_3)^{-c}
\\
\cdot \prod_{j=1}^k\sum_{(c',x,y,z)\in S'_j}c'(xu_1+yu_2+z(\sigma)u_3)^{2+i_j}.
\end{multline*}
Differentiating the above
 product $\kappa$ times with respect to $u_3$ and then setting $u_3$ equal to $0$ is easily done. We obtain
\begin{equation*}
\left(\frac{\partial}{\partial u_3}\right)^\kappa(u_3^{\kappa_0} 
\bw_{i_1,\ldots,i_k}(Q_\sigma)) 
|_{u_3=0}
= \sum_{i\in\mathcal{I}}\phi(\sigma)Z_i(\sigma)R_i(u_1,u_2),
\end{equation*}
where $\mathcal{I}$ is an indexing set, $Z_i\in\Q[\sigma]$ has degree at most $\kappa$, and $R_i(u_1,u_2)\in\Q(u_1,u_2)$ does not depend on $\sigma$.

Proposition~\ref{vanishing} will follow from the claim that
\begin{equation}\label{cancel}
\sum_{\sigma\in\Sym_\mu\text{ admissible}}\phi(\sigma)Z(\sigma) = 0
\end{equation}
for any polynomial $Z$ of degree $\kappa < \kappa_0$ 
(or degree $\kappa=
\kappa_0$ for all but finitely many $f$). The
vanishing property \eqref{cancel} 
is purely a property of the rational function $\phi\in\Q(\sigma)$.

We will now study $\phi$ in more detail.
 The goal is to find a polynomial $\psi\in\Q[\sigma]$ of 
sufficiently low degree satisfying
 $$\phi(\sigma) = \sgn(\sigma)\psi(\sigma)$$ 
for every admissible $\sigma\in\Sym_\mu$ and 
satisfying
$\psi(\sigma) = 0$ for every inadmissible $\sigma\in\Sym_\mu$.
From the formula for $\bV_\sigma$,
we can describe $\phi\in\Q(\sigma)$ explicitly as a product of linear factors:
\footnotesize
\begin{align*}
\phi(\sigma) = 
&\left(\prod_{\substack{(0; j)\in\mu \\ e_0(j)>0}}\sigma_0(j)\right)
\left(\prod_{\substack{(0; j)\in\mu \\ j>0}}j\right)^{-1}
\left(\prod_{\substack{(0; j)\in\mu \\ e_0(j)<-1}}(-\sigma_0(j)-1)\right) 
\\ &
\left(\prod_{\substack{(\delta; j_1),(\delta; j_2)\in\mu \\ e_\delta(j_1)>e_\delta(j_2)}}(\sigma_\delta(j_1)-\sigma_\delta(j_2))\right)^{-1} 
\left(\prod_{\substack{(\delta; j_1),(\delta; j_2)\in\mu \\ j_1>j_2}}(j_1-j_2)\right)
\\ &
\left(\prod_{\substack{(\delta; j_1),(\delta; j_2)\in\mu \\ e_\delta(j_1)>e_\delta(j_2)+1}}(\sigma_\delta(j_1)-\sigma_\delta(j_2)-1)\right)^{-1}
\left(\prod_{\substack{(\delta; j_1),(\delta; j_2)\in\mu \\ j_1>j_2+1}}(j_1-j_2-1)\right) \\
&\left(\prod_{\substack{(\delta+1; j_1),(\delta; j_2)\in\mu \\ e_{\delta+1}(j_1)>e_\delta(j_2)}}(\sigma_{\delta+1}(j_1)-\sigma_{\delta}(j_2))\right)
\left(\prod_{\substack{(\delta+1; j_1),(\delta; j_2)\in\mu \\ j_1>j_2}}(j_1-j_2)\right)^{-1} \\
&\left(\prod_{\substack{(\delta; j_1),(\delta+1; j_2)\in\mu \\ e_{\delta}(j_1)>e_{\delta+1}(j_2)+1}}(\sigma_{\delta}(j_1)-\sigma_{\delta+1}(j_2)-1)\right)
\left(\prod_{\substack{(\delta; j_1),(\delta+1; j_2)\in\mu \\ j_1>j_2+1}}(j_1-j_2-1)\right)^{-1}.
\end{align*}
\normalsize
The degree of $\phi$ is easily computed to be $-\kappa_0$, since there are the same number of constant factors appearing on the numerator and denominator in the above expression.

\begin{lem} \label{frrg} We have
\begin{equation*}
\frac{\prod_{\substack{(\delta; j_1),(\delta; j_2)\in\mu \\ j_1>j_2}}(j_1-j_2)}{\prod_{\substack{(\delta; j_1),(\delta; j_2)\in\mu \\ e_\delta(j_1)>e_\delta(j_2)}}(\sigma_\delta(j_1)-\sigma_\delta(j_2))} = \pm\sgn(\sigma)\prod_{\substack{(\delta; j_1),(\delta; j_2)\in\mu \\ e_\delta(j_1)=e_\delta(j_2) \\ j_1>j_2}}(\sigma_\delta(j_1)-\sigma_\delta(j_2))
\end{equation*}
for every $\sigma\in\Sym_\mu$.
\end{lem}

\begin{proof}
The formula is obtained by cancelling equal terms on the left side.
\end{proof}

Suppose that $\{\delta \mid \mu_\delta \ne \emptyset\} = \{\delta \mid a\le\delta\le b\}$. By using the identity of Lemma \ref{frrg} 
and grouping terms appropriately, we find
$$\phi(\sigma) = \sgn(\sigma)\phi_0(\sigma)$$
for $\phi_0\in\Q(\sigma)$ given by
\begin{equation}
\phi_0 = XPQ\frac{\prod_{a\le \delta\le b-1}R_\delta}{\prod_{a+1 \le \delta \le b-1}S_\delta},
\end{equation}
where
\begin{equation*}
P = \prod_{\substack{j\in\mu_0 \\ e_{0}(j)<-1}}(-\sigma_{0}(j)-1),\ \ \ 
Q = \prod_{\substack{j\in\mu_0 \\ e_{0}(j)>0}}\sigma_{0}(j),
\end{equation*}
\footnotesize
\begin{equation*}
R_\delta = \left(\prod_{\substack{j_1\in\mu_{\delta+1}, j_2\in\mu_{\delta} \\ e_{\delta+1}(j_1)>e_\delta(j_2)}}(\sigma_{\delta+1}(j_1)-\sigma_{\delta}(j_2))\right)\left(\prod_{\substack{j_1\in\mu_{\delta}, j_2\in\mu_{\delta+1} \\ e_{\delta}(j_1)>e_{\delta+1}(j_2)+1}}(\sigma_{\delta}(j_1)-\sigma_{\delta+1}(j_2)-1)\right),
\end{equation*}
\normalsize
\begin{equation*}
S_\delta = \prod_{\substack{j_1,j_2\in\mu_{\delta} \\ e_{\delta}(j_1)>e_\delta(j_2)+1}}(\sigma_{\delta}(j_1)-\sigma_{\delta}(j_2)-1),
\end{equation*}
and $X\in\Q[\sigma]$ is a polynomial. The total degree of the
 rational function $\phi_0$ is
\begin{equation*}
\deg(\phi)+\deg\left(\prod_{\substack{(\delta; j_1),(\delta; j_2)\in\mu \\ j_1>j_2}}(\sigma_\delta(j_1)-\sigma_\delta(j_2))\right) = -\kappa_0 + \sum_{\delta}\frac{1}{2}|\mu_\delta|(|\mu_\delta|-1).
\end{equation*}

We now require an algebraic result in order to
 convert $\phi_0$ into a polynomial.
Let $m,n\geq 0$ be integers, and 
let $$A=\Q[x_1,\ldots,x_n,y_1,\ldots,y_m].$$
Let $P$ be the collection of $n!m!$
points $$(x_1,\ldots,x_n,y_1,\ldots,y_m)\in \Q^{n+m}$$ satisfying
$\{x_1,\ldots,x_n\} = \{1,\ldots,n\}$  and 
$\{y_1,\ldots,y_m\} = \{1,\ldots,m\}$.
 Let $a_1\le a_2\le\cdots\le a_n$ be integers with $0\le a_i < i$, and set
\begin{equation*}
F = \prod_{1\le j\le a_i}(x_j-x_i+1) \in A.
\end{equation*}
The following Proposition will be proven in Section \ref{division}.

\begin{prop}\label{division}
If $G\in A$ vanishes when evaluated at every point of $P$ at which $F$ vanishes,
 then there exists $H\in A$ with $$\deg(H) \le \deg(G)-\deg(F)$$ satisying
 $G=FH$ for every point of $P$.
\end{prop}

If $S_{\delta+1}(\sigma)=0$ for a given $\sigma\in\Sym_\mu$ (which is then
 necessarily inadmissible), then 
\begin{equation}
\label{y34}
R_\delta(\sigma)=R_{\delta+1}(\sigma)=0.
\end{equation} 
By reindexing the permutation sets $\mu_\delta$ and $\mu_{\delta+1}$ as necessary, we can apply Proposition~\ref{division} with $G = R_\delta$ and $F = S_\delta$, since $S_{\delta}$ is of the appropriate form.\footnote{By definition,
 $e_\delta(j)$ is a weakly decreasing function of $j$. We use the {\emph{opposite}} ordering on the variables $\sigma_\delta(j)$ to write $S_\delta$ in the desired form. Explicitly, if $$\mu_\delta = \{A, A+1, \ldots, B\},$$ then we take $x_i = \sigma_\delta(B-i+1)-A+1$.} Thus for 
$a+1\le\delta\le b-1$, there exist polynomials $T_\delta\in\Q[\sigma]$ 
with  $\deg(T_\delta)\le \deg(R_\delta) - \deg(S_\delta)$
satisfying
$$T_\delta(\sigma) = \frac{R_\delta(\sigma)}{S_\delta(\sigma)}$$
for all $\sigma$ for which
 which $S_\delta(\sigma)\neq 0$. Then
\begin{equation*}
\psi = XPQR_a\prod_{a+1 \le \delta \le b-1}T_\delta \in \Q[\sigma]
\end{equation*}
has degree at most equal to that of $\phi_0$ and satisfies 
$$\sgn(\sigma)\psi(\sigma) = \sgn(\sigma)\phi_0(\sigma)=\phi(\sigma)$$
 for any admissible $\sigma$. 

 For a polynomial $\theta\in\Q[\sigma]$,  let $V(\theta)$ 
denote the set of $\sigma\in\Sym_\mu$ such that $\theta(\sigma)=0$. 
We see
\begin{align*}
V(\psi) &\supseteq V(Q)\cup V(R_a) \cup \left(\bigcup_{a+1 \le \delta \le b-1}V(T_\delta)\right) \\
&\supseteq V(Q)\cup V(R_a) \cup \left(\bigcup_{a+1 \le \delta \le b-1}(V(R_\delta)-V(S_\delta))\right) \\
&\supseteq V(Q)\cup \left(\bigcup_{a \le \delta \le b-1}V(R_\delta)\right) \\
&= \{\sigma\in\Sym_\mu \mid \sigma\text{ is not admissible}\}. 
\end{align*}
The third inclusion is by repeated application of \eqref{y34}.
We conclude $\psi$ vanishes when evaluated at any inadmissible $\sigma$.

We are finally able to evaluate the sum \eqref{cancel}. We have
\begin{equation*}
\sum_{\sigma\in\Sym_\mu\text{ admissible}}\phi(\sigma)Z(\sigma) = \sum_{\sigma\in\Sym_\mu}\sgn(\sigma)\psi(\sigma)Z(\sigma).
\end{equation*}
If $\deg(Z)<\kappa_0$, then $\deg(\psi Z)<\sum_{\delta}\frac{1}{2}|\mu_\delta|(|\mu_\delta|-1)$, and thus 
\begin{equation*}
\sum_{\sigma\in\Sym_\mu}\sgn(\sigma)\psi(\sigma)Z(\sigma) = 0.
\end{equation*} 
We have proven the evaluation of Proposition \ref{vanishing}
 is well-defined.

The second part of Proposition~\ref{vanishing} asserts the
vanishing of the evaluation for all but finitely many $f$. We will use a combination of two 
ideas to prove the assertion. 
First, if $\SSS(f) = \emptyset$, then the evaluation is trivially zero. 
Second, we replace
the polynomial $\psi$ above with another polynomial
 $\psi'$ which assumes the same values but has lower degree. Then 
$$\deg(\psi')<\deg(\phi_0) 
= -\kappa_0 + \sum_{\delta}\frac{1}{2}|\mu_\delta|(|\mu_\delta|-1).$$
So for $\deg(Z)\leq\kappa_0$,
\begin{equation*}
\sum_{\sigma\in\Sym_\mu}\sgn(\sigma)\psi'(\sigma)Z(\sigma) = 0.
\end{equation*} 

As we have seen, a choice of $f$ such that $\SSS(f)\ne\emptyset$ uniquely determines 
constants $e_\delta(j)$ weakly decreasing in $j$. We use linear inequalities in 
the constants $e_\delta(j)$ 
to describe four cases in which either $\SSS(f) = \emptyset$ or $\psi$ can be 
replaced by $\psi'$ as above. In the end, we will check that only finitely many 
possibilities avoid all four cases. The finiteness will come from giving upper 
and lower bounds for the $e_\delta(j)$. 
For the lower bound, since
$e_\delta(j)$ is weakly decreasing in $j$, we introduce the notation
\[
m_\delta = \max(\mu_\delta)
\]
and focus on the values $e_\delta(m_\delta)$.

\vspace{10pt}
\noindent {\bf Case I.}
Let $J = \max\{j\mid (\delta; j)\in\mu\text{ for some }\delta\}$ and suppose 
 $e_\delta(j) > J$ for some $(\delta; j)\in\mu$.
Then for any $\sigma\in\Sym_\mu$,
\[
h_\sigma(\delta; \sigma_\delta(j)) = a\cdot(\sigma_\delta(j)-e_\delta(j)) < 0,
\]
so $\sigma$ is not admissible. Thus $\SSS(f)=\emptyset$.

\vspace{10pt}
\noindent {\bf Case II.}
Consider the sequence 
$$e_0(0)\ge e_0(1)\ge \cdots \ge e_0(m_0).$$ 
Suppose there exists  $i\in\{0,\ldots,m_0\}$ for which 
the conditions
\begin{enumerate}
\item[$\bullet$]
$e_0(i)<-1$
\item[$\bullet$]
$i=0$ or $e_0(i)<e_0(i-1)-1$
\end{enumerate}
hold.
Then, for
admissible $\sigma\in\Sym_\mu$, the factor $\sigma_0$ must map
$\{i,\ldots,m_0\}$ to itself, as the box configuration function 
$$h_\sigma(\delta; j) = a(j-e_\delta(\sigma_\delta^{-1}(j)))$$ 
must be weakly increasing in $j$. 
The factor $P$ of $\psi$ is a multiple of
\begin{equation*}
\prod_{j=i}^{m_0}(-\sigma_0(j)-1).
\end{equation*}
Since $\frac{\psi}{P}$ vanishes at all inadmissible $\sigma$, we can take
\begin{equation*}
\psi' = \frac{\prod_{j=i}^{m_0}(-j-1)}{\prod_{j=i}^{m_0}(-\sigma_0(j)-1)}\psi,
\end{equation*}
and then $\psi'(\sigma)=\psi(\sigma)$ at all $\sigma\in\Sym_\mu$.
We have $\deg(\psi')<\deg(\psi)$, as desired.

\vspace{10pt}
\noindent {\bf Case III.}
Suppose $\delta \ge 0$ and $e_{\delta+1}(m_{\delta+1}) +1 < e_{\delta}(m_{\delta})$.

Then, either $m_{\delta+1} = m_{\delta} - 1$ or $m_{\delta+1} = m_{\delta}$. We consider the two options separately.

\vspace{10pt}
\noindent({\bf{i}}) 
If $m_{\delta+1} = m_{\delta} - 1$, then for any $\sigma\in\Sym_\mu$, we can take
$$i = \sigma_\delta^{-1}(\sigma_{\delta+1}(m_{\delta+1})+1)\ .$$ Then,
$
\sigma_\delta(i) = \sigma_{\delta+1}(m_{\delta+1})+1$
and  
$e_\delta(i) \ge e_\delta(m_\delta) > e_{\delta+1}(m_{\delta+1})+1$,
so $\sigma$ is not admissible. Thus $\SSS(f)=\emptyset$.

\vspace{10pt}
\noindent({\bf{ii}})
If $m_{\delta+1} = m_{\delta}$, then we have 
 $e_{\delta+1}(m_{\delta}) +1 < e_{\delta}(m_{\delta}) \le e_\delta(j)$
 for $0\le j \le m_\delta$, so $R_\delta$ is a multiple of
\begin{equation}\label{jj45}
\prod_{j=0}^{m_\delta}(\sigma_\delta(j)-\sigma_{\delta+1}(m_\delta)-1).
\end{equation}
The product \eqref{jj45} vanishes unless $\sigma_{\delta+1}(m_\delta)=m_\delta$.
Hence
\begin{equation*}
(-m_\delta-1)\prod_{j=1}^{m_\delta}(j-\sigma_{\delta+1}(m_\delta)-1)
\end{equation*}
equals \eqref{jj45} 
for all $\sigma\in\Sym_\mu$ and is of lower degree, so we may replace $\psi$ with $\psi'$ of lower degree.

\pagebreak
\noindent {\bf Case IV.}
Suppose $\delta < 0$ and $e_{\delta}(m_\delta) < e_{\delta+1}(m_{\delta+1})$.

The situation is parallel to Case III. As before, either $\SSS(f)=\emptyset$ or we can replace a divisor of $R_\delta$ with a polynomial of lower degree.

\vspace{10pt}
To complete the proof of Proposition \ref{vanishing}, we
must check there are only 
finitely many $f$ 
which avoid Cases I-IV.
If $f$ does not fall into Case I, then $e_\delta(j)\le J$ for all $(\delta; j)\in\mu$. If $f$ does not fall into Case II, then $e_0(j)\ge -j-1$ for each $j$, and in particular $e_0(m_0) \ge -m_0-1$. If $f$ also does not fall 
into either of the other two cases, 
we can extend the inequality to obtain 
$$e_\delta(m_\delta) \ge -m_0 - 1 - \max\{\delta \mid \mu_\delta\ne\emptyset\}$$
for all $\delta$.
Since $e_\delta(j)$ is a weakly decreasing function of $j$, the
 bounds imply bounds for all of the $e_\delta(j)$. Since the $e_\delta(j)$ belong to $\frac{1}{a}\Z$, we conclude
 there are only a finite number of possibilities for each if $f$ does not fall into any of the Cases I-IV. \qed

\subsection{Proof of Proposition~\ref{division}}
Let $R=\Q[x_1,\ldots,x_n]$, and let 
$$e_1, e_2,\ldots,e_n\in R$$ be the elementary 
symmetric polynomials with $c_1, c_2, \ldots, c_n \in \Z$ their evaluations at $x_i=i$. 
Let $$I = (e_1-c_1,\ldots,e_n-c_n) \subset R$$ denote the ideal of polynomials 
vanishing on every permutation of $(1,\ldots,n)$. 
For a polynomial $f\in R$, let $f_0$ denote the homogeneous part
of $f$ of highest degree.
For  an ideal $J\subset R$, let $J_0$ denote the homogeneous 
ideal generated by the top-degree parts,
$$J_0 = \langle \ f_0  \ | \ f \in J \ \rangle \ .$$
Using the regularity of $e_1,\ldots, e_n$,
we easily see $I_0 = (e_1,\ldots,e_n)$.

We define $R'=\Q[y_1,\ldots,y_m]$ and ideals $I',I'_0\subset R'$
as above with respect to the permutations of $(1,\ldots,m)$.
We have
 $$A = R\otimes_\Q R'= \Q[x_1,\ldots, x_n,y_1,\ldots,y_m]\ .$$
For notational convenience, we let
$$I,I_0,I',I'_0 \subset A$$
denote the extensions of the respective ideals of $R$ and $R'$ in $A$. 
The ideal of $A$ vanishing on the set $P\subset \Q^{n+m}$ 
of Proposition \ref{division} is precisely $I+I'$.
The basic equality
$$(I+I')_0 = I_0 + I_0'$$
holds.

Let $\widehat{P} = \{p\in P \mid F(p)\ne 0\}$. Let $H\in A$ be a polynomial with 
the prescribed values 
$$H(p) = \frac{G(p)}{F(p)}$$ for  
$p\in \widehat{P}$, of minimum possible degree $d = \deg(H)$. 
We must show $d\le \deg(G)-\deg(F)$.
For contradiction, assume  $d > \deg(G)-\deg(F)$. 
Then, the polynomial $G - FH$ vanishes at every $p\in P$ and has top degree part $F_0H_0$.

Since $F_0\in R$, we verify the following equality
\begin{equation*}
H_0\in\{f\in A \mid F_0f \in (I+I')_0\} = \{r\in R \mid F_0r \in I_0\} + I'_0 \ \ \subset A.
\end{equation*}
We claim the above ideal is equal to 
\begin{equation*}
\{f\in A \mid Ff \in I+I'\}_0 = \{r\in R \mid Fr \in I\}_0 + I'_0 \ \ \subset A,
\end{equation*}
Assuming the equality, 
there exists $H'\in A$ with top degree part $H_0$ and $FH' \in I+I'$ vanishing at 
every $p\in P$. 
But then $H_0-H'$ has degree less than that of $H_0$ and still interpolates the desired values, so we have a contradiction.

To complete the proof of Proposition \ref{division},
we must show 
\begin{equation*}
\{r\in R \mid F_0r \in I_0\} + I'_0 = \{r\in R \mid Fr \in I\}_0 + I'_0,
\end{equation*}
or equivalently 
\begin{equation} \label{befff}
\{r\in R \mid F_0r \in I_0\} = \{r\in R \mid Fr \in I\}_0.
\end{equation}
The left hand side contains the right hand side.
The equality \eqref{befff} is thus a 
 consequence of the following Lemma which implies 
the two sides have equal (and finite) codimension in $R$.

\begin{lem}\label{ranks}
Let $n\geq 0$ be an integer, and 
let $a_1\le a_2\le\cdots\le a_n$ be integers satisfying $0\le a_i < i$.  Let
\begin{equation*}
F = \prod_{1\le j\le a_i}(x_j-x_i+1)  \ \ \ \ {\text and} \ \ \ \
F_0 = \prod_{1\le j\le a_i}(x_j-x_i).
\end{equation*}
Then, we have
\begin{eqnarray*}
\rk_\Q(m_F: R/I \to R/I) & =& \rk_\Q(m_{F_0}:R/I_0 \to R/I_0) \\
                &  = & \prod_{i=1}^n(i-a_i),
\end{eqnarray*}
where $m_F$ and $m_{F_0}$ denote multiplication operators by $F$ and $F_0$ respectively.
\end{lem}
\begin{proof}
We first show  $\rk_\Q(m_F) = \prod_{i=1}^n(i-a_i)$. 
Since $R/I$ is the coordinate ring of the set of $n!$ permutations of $(1,\ldots,n)$, 
the rank is simply the number of permutations at which $F$ does not vanish. 
We must  count the number of permutations $\sigma\in \Sym_n$ 
satisfying $$\sigma(i)-1\ne\sigma(j)$$
 for $1\le j\le a_i$. 

We view the permutation $\sigma$  (extended by $\sigma(0)=0$)
 as a directed path on vertices labeled $0,1,\ldots,n$ with an edge from $i$ to $j$ 
if $\sigma(i)-1=\sigma(j)$. 
We are then counting permutations which
 do not have an edge from $i$ to $j$ if $1\le j\le a_i$.

We count the number of ways of building such a path by first choosing an edge leading out of $n$, then an edge leading out of $n-1$, and so on. The edge leading out of $n$ can go to $0$ or to any $j$ with $a_n < j < n$; there are $n-a_n$ choices. After placing the edges leading out of $n,n-1,\ldots,k+1$, the digraph will be a disjoint union of $k+1$ paths. One of these paths will end at $k$ and $a_k$ of the other paths will end at $1,\ldots,a_k$, so the choices for the edge leading out of $k$ are to go to the start of one of the $k-a_k$ other paths. Thus,
 the number of such permutations is indeed the product $(n-a_n)\cdots(1-a_1)$.

Proving $\rk_\Q(m_{F_0}) = \prod_{i=1}^n(i-a_i)$ will require more work. 
Let $$J = \{f\in R \mid F_0f\in I_0\},$$
 so multiplication by $F_0$ induces an isomorphism between $R/J$ and  
$\text{Image}(m_{F_0})\subset R/I_0$. 
We will show 
\begin{equation}\label{uu23}
\rk_\Q(R/J) = \prod_{i=1}^n(i-a_i)\ .
\end{equation}  
In fact, we claim  $R/J$ is a 0-dimensional complete intersection of 
multidegree $(1-a_1,\ldots,n-a_n)$. The
dimension \eqref{uu23}
will then follow from Bezout's Theorem. 
For $1\le k \le n$, let
\begin{equation*}
f_k = \sum_{i=k}^n x_i\prod_{j=a_k+1}^{k-1}(x_j-x_i).
\end{equation*}
We claim $J = (f_1,\ldots,f_n)$. Note $f_k$ has degree $k-a_k$ as desired.

We will prove this claim by induction on the sequence $(a_i)_{i=1}^n$. The base case is $a_i=0$ for all $i$
where
 $$F=1,\ \ \ J = I_0,\ \ \ \text{and}\ \ \ 
f_k = \sum_{i=k}^nx_i\prod_{j=1}^{k-1}(x_j-x_i)\ .$$
 We must show $(f_1,\ldots,f_n) = (e_1,\ldots,e_n)$.

First, suppose $f_1=f_2=\cdots=f_n=0$ 
at some point $$(t_1,\ldots,t_n)\in\overline{\Q}^n.$$
 From $f_n=0$, we find either $t_n=0$ or $t_n=t_i$ for some $i<n$. 
Since $f_{n-1}=0$, either $t_{n-1}=0$ or $t_{n-1}=t_i$ for some $i<n-1$. 
Continuing, we conclude for every $k$, either $t_k=0$ or $t_k=t_i$ for some $i<k$. 
Thus, $t_k=0$ for all $k$. 
Therefore $R/(f_1,\ldots,f_n)$ is a complete intersection and has $\Q$-rank 
$$(\deg f_1)\cdots(\deg f_n) = n! = \rk_\Q(R/(e_1,\ldots,e_n))\ .$$

By the rank computation, we
need only show  
\begin{equation}\label{httyy}
(f_1,\ldots,f_n) \subseteq (e_1,\ldots,e_n)
\end{equation}
to complete the base case of the induction. 
But the inclusion \eqref{httyy} is easily seen. For every $k$, we have
\begin{align*}
f_k &= \sum_{i=1}^n x_i\prod_{j=1}^{k-1}(x_j-x_i) \\
&= \sum_{i=1}^n \sum_{e=1}^n c_ex_i^e \\
&=\sum_{e=1}^n c_e \left(\sum_{i=1}^n x_i^e\right),
\end{align*}
where $c_e\in R$. The power sum
 $\sum_{i=1}^n x_i^e$ is symmetric and can be written as a polynomial in the elementary symmetric functions $e_1,\ldots,e_n$.  
The base case is now established.

We now consider two sets of indices $a_1,\ldots, a_n$ and $a'_1,\ldots,a'_n$
for which such that $a'_i=a_i$ except when $i=l$ and 
\begin{equation}
a'_l=a_l+1.
\label{grtt}
\end{equation}
 We moreover require either $l = n$ or $a_{l+1}=a_l+1$. We assume 
inductively our claim holds for $a_1,\ldots, a_n$ and show 
the claim for $a'_1,\ldots,a'_n$. 
Every $(a'_i)_{1\le i \le n}$ which is not identically zero can be reached by 
taking $l = \min\{l \mid a'_l=a'_n\}$, so the inductive step will imply the Lemma. 
Let $J$,$J'$ be the corresponding ideals and let $f_1,\ldots,f_n$ and $f'_1\ldots,f'_n$ be the claimed generators.
 We are assuming  $J = (f_1,\ldots,f_n)$ and want to prove  $J' = (f'_1\ldots,f'_n)$.

 From the definition of $J$ and $J'$, we easily see 
 $$J' = \{g\in R \mid (x_{a_l+1}-x_l)g\in J\}.$$
 Also note $f'_k = f_k$ for $k\ne l$.
If $l=n$, then
 $$f'_l = \frac{f_l}{x_{a_l+1}-x_l}$$ 
and otherwise
$$f'_l = \frac{f_l-f_{l+1}}{x_{a_l+1}-x_l}$$
by condition \eqref{grtt}.

Let $\overline{R} = R/(x_{a_l+1}-x_l)$. For an element $r\in R$, 
let $\overline{r}$ denote the projection in $\overline{R}$. Consider 
the $\overline{R}$-module homomorphism 
$$\psi: \overline{R}^n\to \overline{R}$$
 defined by $\psi(\overline{r}_1,\ldots,\overline{r}_n) = \overline{f}_1\overline{r}_1+\cdots+\overline{f}_n\overline{r}_n$. Let $s_i^{(j)}$ for $1\le i\le n$ and $1\le j\le m$ be such that the $m$ elements $(\overline{s}_1^{(j)},\ldots, \overline{s}_n^{(j)})\in\overline{R}^n$ generate the kernel of $\psi$. Clearly,
 $J'$ is the ideal generated by $J$ and the $m$ elements 
$$\frac{1}{x_{a_l+1}-x_l}\sum_{i=1}^nf_is^{(j)}_i \ .$$

In other words, we must find all the relations between the elements 
$$\overline{f}_1,\ldots,\overline{f}_n.$$
 Now $\overline{f}_l = \overline{f}_{l+1}$ if $l\ne n$, or $\overline{f}_l=0$ if $l = n$, so we need only consider relations between the $n-1$ elements with $\overline{f}_l$ removed. 
These $n-1$ elements in $\overline{R}$ form a complete intersection, 
so the relations are generated by the trivial ones 
$\overline{f}_i\overline{f}_j-\overline{f}_j\overline{f}_i = 0$.

We have proven that
 $J'$ is the ideal generated by $J = (f_1,\ldots,f_n)$, either $\frac{f_l}{x_{a_l+1}-x_l}$ if $l=n$ or $\frac{f_l-f_{l+1}}{x_{a_l+1}-x_l}$ otherwise, and the elements $$\frac{f_if_j-f_jf_i}{x_{a_l+1}-x_l} = 0.$$
 Thus $J' = (f'_1\ldots,f'_n)$, as desired.
\end{proof}

\section{Descendent depth}\label{depp}

\subsection{$T$-Depth} 
Let $N$ be a split rank 2 bundle on 
a nonsingular projective curve $C$ of genus $g$.
Let $S\subset N$ be the relative divisor associated
to the points $p_1,\ldots, p_r\in C$.
We consider the $T$-equivariant stable pairs theory of $N/S$
with respect to the scaling action.

The $T$-{\em depth} $m$ theory of $N/S$ consists of all
$T$-equivariant series
\begin{equation}
\label{hkkq}
{\mathsf Z}^{N/S}_{d,\eta^1,\dots,\eta^r} 
\left( \prod_{{j'}=1}^{k'}
\tau_{i'_{j'}}(\mathsf{1})  \  \prod_{j=1}^k \tau_{i_j}(\mathsf{p})   
\right)^T
\end{equation}
where $k' \leq m$.
As before, $\mathsf{p}\in H^2(C,\mathbb{Z})$ is the class of a point.
The $T$-depth $m$ theory has at most $m$ descendents
of $1$ and arbitrarily many descendents of $\mathsf{p}$ in the integrand.
The $T$-depth $m$ theory of $N/S$ is {\em rational} if all
$T$-depth $m$ series \eqref{hkkq} are Laurent expansions in $q$
of rational functions in $\Q(q,s_1,s_2)$.

The $T$-depth 0 theory concerns only descendents of $\mathsf{p}$.
By taking the specialization $s_3=0$ of Proposition \ref{cttt}, 
$$
\ZZ_{d,\eta}
^{\mathsf{cap}}\left(   \prod_{j=1}^k \tau_{i_j}(\mathsf{p})
\right)^T=
\ZZ_{d,\eta}
^{\mathsf{cap}}\left(   \prod_{j=1}^k \tau_{i_j}([0])
\right)^{\mathbf{T}}\Big|_{s_3=0}\ , $$
we see
the depth 0 theory of the cap is rational.  

\begin{lem} The $T$-depth 0 theory of $N/S$ over a curve $C$ is rational.
\label{ht99}
\end{lem}

\begin{proof} By the degeneration formula, all the descendents
$\tau_{i_j}(\mathsf{p})$ can be degenerated on to a $(0,0)$-cap. 
The $T$-depth 0 theory of the cap is rational.
The pairs
theory of local curves without any insertions is rational by \cite{mpt,lcdt}.
Hence, the result follows by the degeneration formula. 
\end{proof}

\subsection{Degeneration}
We have already used the degeneration formula in simple cases 
in Proposition \ref{ctttt}
and Lemma  \ref{ht99} above. We review here the
full $T$-equivariant formula for descendents of 
$\mathsf{1},\mathsf{p}\in H^*(C,\mathbb{Z})$.

Let $C$ degenerate to a union $C_1\cup C_2$ of nonsingular
projective curves $C_i$
meeting at a node $p'$. 
Let $N$ degenerate to split bundles 
$$N_1 \rightarrow C_1, \ \ \ \ N_2 \rightarrow C_2 \ .$$
The levels of $N_i$ must sum to the level of $N$.
The relative points $p_i$, distributed to nonsingular points
of $C_1\cup C_2$, specify relative points $S_i\subset C_i$ away from
$p'$. Let $S_i^+= S_i \cup \{ p' \}$.

In order to apply the degeneration formula to the series
\eqref{hkkq}, we must also specify the distribution of
the point classes occuring in the descendents $\tau_{i_j}(\mathsf{p})$.
The disjoint union $$J_1\cup J_2 = \{1,\ldots, k\}$$
specifies the descendents $\tau_{i_j}(\mathsf{p})$
distribute to $C_i$ for $j\in J_i$.
The degeneration formula for \eqref{hkkq} is
\begin{multline*}
\sum_{J'_1\cup J'_2=\{1,\ldots, k'\}}
{\mathsf Z}^{N_1/S^+_1}_{d,\eta^1,\dots,\eta^{|S_1|}, \mu} 
\left( \prod_{{j'}\in J'_1}
\tau_{i'_{j'}}(\mathsf{1})  \  \prod_{j\in J_1} \tau_{i_j}(\mathsf{p})   
\right)^T\ \frac{g^{\mu\widehat{\mu}}}{q^d} \\ \cdot
{\mathsf Z}^{N_2/S^+_2}_{d,\eta^{|S_1|+1},\dots,\eta^{|S_2|}, \widehat{\mu}} 
\left( \prod_{{j'}\in J'_2}
\tau_{i'_{j'}}(\mathsf{1})  \  \prod_{j\in J_2} \tau_{i_j}(\mathsf{p})   
\right)^T
\end{multline*}

A crucial point in the derivation of the degeneration formula
is the pre-deformability condition (ii) of Section 3.7 of \cite{pt}.
The condition insures the existence of finite resolutions of
the universal sheaf $\mathbb{F}$ in the relative geometry (needed
for the definition of the descendents) and guarantees the
splitting of the descendents under pull-back via the gluing 
maps of the relative geometry. The foundational treatment for
stable pairs is essentially the same as for ideal sheaves \cite{liwu}.

\subsection{Induction I}

To obtain the rationality of the $T$-depth $m$ theory of
$N/S$ over a curve $C$, further knowledge of the descendent
theory of twisted caps is required.

\begin{lem} The rationality of the
 $T$-depth $m$ theories of all twisted caps implies
the rationality of the 
$T$-depth $m$ theory
of $N/S$ over a curve $C$. \label{rtt5}
\end{lem}

\begin{proof}
We start by proving rationality for the $T$-depth $m$ theories of
all $(0,0)$ geometries,
\begin{equation} \label{hxxz}
\cO_{\C} \oplus \cO_{\C} \rightarrow \Pp\ ,
\end{equation}
relative to $p_1,\ldots, p_r \in \Pp$.
If $r=1$, the geometry is the cap and rationality of the
$T$-depth $m$ theory is given.
Assume rationality holds for $r$. We will show rationality holds for
$r+1$.

Let $p(d)$ be the number of partitions of size $d>0$.
Consider the $\infty \times p(d)$ matrix $M_d$, indexed by 
monomials 
$$L= \prod_{i\geq 0} \tau_i (\mathsf{p})^{n_i} $$
in the descendents of $\mathsf{p}$ and partitions $\mu$ of $d$,
 with
coefficient 
$
{\mathsf Z}^{\mathsf{cap}}_{d,\mu} 
\left( L   
\right)^T$
in position $(L,\mu)$.
The lowest Euler characteristic for a degree $d$
stable pair on the cap is $d$.  
The leading $q^d$
coefficients of $M_d$ are well-known to be of maximal
rank.{\footnote{ The leading $q^d$ coefficients
are obtained from the Chern characters
of the tautological rank $d$ bundle
on $\text{Hilb}(N_\infty,d)$.
The Chern characters generate the ring
$H^*_T(\text{Hilb}(N_\infty,d),\mathbb{Q})$ after
localization as can easily
be seen in the $T$-fixed point basis. 
A more refined result is
discussed in Section \ref{ennd}.}}
Hence, the full matrix $M_d$ is also of maximal rank.

Consider the level $(0,0)$ geometry 
over $\Pp$ relative to $r+1$ points in $T$-depth $m$,
\begin{equation}
\label{yone}
{\mathsf Z}^{(0,0)}_{d,\eta^1, \ldots, \eta^r,\mu} 
\left( \prod_{j'=1}^{k'}
\tau_{i'_{j'}}(1)  \  \prod_{j=1}^k \tau_{i_j}(\mathsf{p})   
\right)^T\ .
\end{equation}
We will determine the series \eqref{yone} from the
$T$-depth $m$ series relative to  $r$ points,
\begin{equation}
\label{yall}
{\mathsf Z}^{(0,0)}_{d,\eta^1, \ldots, \eta^r}
\left( L \ \prod_{j'=1}^{k'}
\tau_{i'_{j'}}(1)  \  \prod_{j=1}^k \tau_{i_j}(\mathsf{p})   
\right)^T
\end{equation}
defined by all monomials $L$  in the descendents of $\mathsf{p}$.

Consider the $T$-equivariant degeneration of 
 the $(0,0)$ geometry relative to $r$ points obtained
by bubbling off a single $(0,0)$-cap.
All the descendents of $\mathsf{p}$
remain on the original $(0,0)$ geometry in the degeneration except for those 
in $L$ which distribute to the cap. 
By induction on $m$, we need only analyze the terms of the degeneration formula
in which the descendents of the identity distribute away from the cap.
Then,
since $M_d$ has full rank,
the invariants \eqref{yone} are determined
by the invariants \eqref{yall}.

We have proven the rationality of the $T$-depth $m$ theory of the $(0,0)$-cap
implies the rationality of the $T$-depth $m$ theories of all
$(0,0)$ relative geometries over $\Pp$. By degenerations of
higher genus curves $C$ to rational curves with relative points,
the  rationality of the $(0,0)$ relative geometries over
curves $C$ of arbitrary genus is established.

Finally, consider a relative geometry $N/S$ over $C$ of
level $(a_1,a_2)$. We can degenerate $N/S$ to the
union of a $(0,0)$ relative geometry over $C$ and 
 a twisted $(a_1,a_2)$-cap. Since the rationality of the
$T$-depth $m$ theory of
the twisted cap is given, we conclude the rationality of
$N/S$ over $C$.
\end{proof}

The proof of Lemma \ref{rtt5} yields a slightly refined
result which will be  half of our induction argument
relating the descendent theory of the $(0,0)$-cap and
the $(0,0)$-tube.

\begin{lem}\label{nndd}
The rationality of the $T$-depth $m$ theory of the 
$(0,0)$-cap implies the rationality of the $T$-depth $m$ theory of
the $(0,0)$-tube. 
\end{lem}

\subsection{$\mathbf{T}$-depth}

The $\mathbf{T}$-{\em depth} $m$ theory of the $(a_1,a_2)$-cap 
consists of all the
$\mathbf{T}$-equivariant series
\begin{equation}
\label{hkkqq}
{\mathsf Z}^{(a_1,a_2)}_{d,\eta} 
\left( 
\prod_{j=1}^k \tau_{i_j}([0])  \  \prod_{{j'}=1}^{k'}
\tau_{i'_{j'}}([\infty])  
\right)^{\mathbf{T}}
\end{equation}
where $k' \leq m$.
Here, $0\in \Pp$ is the non-relative $\mathbf{T}$-fixed point and
$\infty\in \Pp$ is the relative point.
The $\mathbf{T}$-depth $m$ theory of the $(a_1,a_2)$-cap 
is {\em rational} if all
$\mathbf{T}$-equivariant
depth $m$ series \eqref{hkkq} are Laurent expansions in $q$
of rational functions in $\Q(q,s_1,s_2,s_3)$.

\begin{lem} The rationality of the $\mathbf{T}$-depth $m$
theory of the $(a_1,a_2)$-cap implies the
rationality of the $T$-depth $m$ theory of the $(a_1,a_2)$-cap.
\end{lem}

\begin{proof}
The identity class
$1\in H^*_T(\Pp,\mathbb{Z})$ has a well-known
expression in terms of the $\mathbf{T}$-fixed point classes
$$1 = -\frac{[0]}{s_3} + \frac{[\infty]}{s_3}\ .$$
We can calculate at most $m$ descendents of $1$ in  the $T$-equivariant
theory via at most 
$m$ descendents of $[\infty]$ in the $\mathbf{T}$-equivariant
theory (followed the specialization $s_3=0$).
\end{proof}

\section{Rubber calculus} \label{rubc}

\subsection{Overview}
We collect here results concerning 
the rubber calculus which will be needed to complete the
proof of Theorem \ref{cnnn}. Our discussion of the
rubber calculus follows the treatment given in Section 4.8-4.9 of
\cite{lcdt}.

\subsection{Universal 3-fold $\mathcal{R}$}
Consider the moduli space of stable pairs on rubber 
$P_n(R/R_0\cup R_\infty)^\sim$ discussed in Section \ref{rubcon}.
Let  
$$\pi:\mathcal{R} \rarr {P_n(R/R_0\cup R_\infty,d)}^\sim$$
denote the universal 3-fold.
The space $\mathcal{R}$ can be viewed as a moduli space of stable pairs on
rubber {\em together} with a point $r$ of the 3-fold rubber.
The point $r$ is {\em not} permitted to lie on the relative divisors $R_0$ and
$R_\infty$. The stability condition is given by finiteness of the
associated automorphism group.
The virtual class of ${\mathcal R}$ is obtained via $\pi$-flat pull-back,
$$[{\mathcal R}]^{vir} = \pi^* \Big( [{P_n(R/R_0\cup R_\infty,d)}^\sim]^{vir}\Big).$$
As before, let
$$\mathbb{F} \rarr {\mathcal R}$$
denote the universal sheaf on ${\mathcal R}$.

The target point $r$ together with $R_0$ and $R_\infty$
specifies 3 distinct points of the destabilized $\Pp$ over which the rubber is fibered.
By viewing the target point as $1\in \Pp$,
we obtain a rigidification map to the tube, 
$$\phi: \mathcal{R} \rarr P_n(N/N_0\cup N_\infty,d),$$
where $N=\cO_\Pp \oplus \cO_\Pp$ is the trivial bundle over $\Pp$. 
By a comparison of deformation theories,
\begin{equation}\label{zek}
[{\mathcal R}]^{vir} = \phi^* \Big( [{P_n(N/N_0\cup N_\infty,d)}]^{vir}\Big).
\end{equation}

\subsection{Rubber descendents} \label{papap}
Rubber calculus transfers $T$-equivariant rubber descendent
integrals to $T$-equivariant descendent integrals
for the $(0,0)$-tube geometry via the maps $\pi$ and $\phi$.
Consider the rubber descendent
\begin{equation} \label{dref} 
\Big\langle \mu \ \Big| \ \psi_0^\ell \ \tau_{c}\cdot \prod_{j=1}^k \tau_{i_j} 
\ \Big|\ \nu \Big
\rangle_{n,d}^{\sim}\ .
\end{equation}
As before, $\psi_0$ is the cotangent line at the dynamical point
$0\in \Pp$. The action of the rubber descendent $\tau_{i}$ is defined
via the universal sheaf $\mathbb{F}$ by the operation
$$
\pi_{*}\big( \text{ch}_{2+i}(\FF)
\cap(\pi^*(\ \cdot\ )\big)\colon 
H_*(P_{n}(N/N_0\cup N_\infty,d))\to H_*(P_{n}(N/N_0\cup N_\infty,d))\ .
$$
By the push-pull formula, the integral \eqref{dref} equals
\begin{equation}\label{zex}
\Big\langle \mu \ \Big|\ \text{ch}_{2+c}(\mathbb{F}) \ \pi^*\left(\psi_0^\ell
 \cdot \prod_{j=1}^k \tau_{i_j}
\right) \ \Big|\ \nu \Big
\rangle_{n,d}^{{\mathcal R}\sim}.
\end{equation}

Next,
we compare the cotangent lines $\pi^*(\psi_0)$ and $\phi^*(\psi_0)$ on
${\mathcal R}$. A standard argument yields
$$\pi^*(\psi_0)=\phi^*(\psi_0) - \phi^*(D_0),$$
where
 $$D_0 \subset I_n(N/N_0\cup N_\infty,d)$$
is the virtual boundary divisor
for which the rubber over $\infty$ carries Euler characteristic $n$.
We will apply the cotangent line comparisons to  \eqref{zex}.
The basic vanishing
\begin{equation}\label{gbb6}
\psi_0|_{D_0}=0
\end{equation}
holds.

Consider the Hilbert scheme of points ${\text{Hilb}}(R_0,d)$ of the
relative divisor.
The boundary condition $\mu$ corresponds to a Nakajima basis element
of $A^*_T({\text{Hilb}}(R_0,d))$.
Let
${\mathbb{F}}_0$ be the universal quotient sheaf on 
$$\text{Hilb}(R_0,d) \times R_0,$$ and define the descendent
\begin{equation}\label{mrr}
\tau_c=\pi_*\Big( {\text{ch}}_{2+c}({\mathbb{F}}_0)\Big) 
\in A^c_T({\text{Hilb}}(R_0,d))
\end{equation}
where $\pi$ is the projection
$$\pi: \text{Hilb}(R_0,d) \times R_0 \rightarrow
\text{Hilb}(R_0,d)\ .
$$

The cotangent line comparisons, equation \eqref{zex},
and the vanishing \eqref{gbb6} together yield
the following result,
\begin{multline}\label{dx}
\Big\langle \mu \ \Big| \   \psi_0^\ell \ \tau_{c}\cdot \prod_{j=1}^k \tau_{i_j}
\ \Big|\ \nu \Big\rangle_{n,d}^{\sim} =
\\ \Big\langle \mu \ \Big|\ 
\psi_0^\ell \ \tau_{c}(\mathsf{p}) \cdot \prod_{j=1}^k \tau_{i_j}
\ \Big|\ \nu \Big
\rangle_{n,d}^{\mathsf{tube},T} \\ 
 - \Big\langle \tau_c\cdot \mu \ \Big| 
\ \psi_0^{\ell-1}  \prod_{j=1}^k \tau_{i_j}     \ \Big|\ \nu 
\Big\rangle_{n,d}^{\sim} \ .
\end{multline}
Equation \eqref{dx} will be the main required property of the rubber
calculus.

\section{Capped 1-leg descendents: full} \label{444}

\subsection{Overview} We complete the proof of Theorem \ref{cnnn}
using the interplay between the $\mathbf{T}$-equivariant localization of the cap
and the theory of rubber integrals. A similar strategy was used in \cite{vir}
to prove the Virasoro constraints for target curves.
As a consequence, we will also obtain  a special case of Theorem 
\ref{tnnn}.

Let $N$ be a split rank 2 bundle on 
a nonsingular projective curve $C$ of genus $g$.
Let $S\subset N$ be the relative divisor associated
to the points $p_1,\ldots, p_r\in C$.
We consider the $T$-equivariant stable pairs theory of $N/S$
with respect to the scaling action.

\begin{prop}
 If
$\gamma_j \in H^{2*}(C,\mathbb{Z})$ are {\em even} cohomology classes, then
\label{pnnn} 
$$\ZZ_{d,\eta^1,\dots,\eta^r}
^{N/S}\left(   \prod_{j=1}^k \tau_{i_j}(\gamma_{j})
\right)^{{T}}$$ is the 
Laurent expansion in $q$ of a rational function in $\mathbb{Q}(q,s_1,s_2)$.
\end{prop}

Proposition \ref{pnnn} is the restriction of Theorem \ref{tnnn}
to even cohomology. The proof is given in Section \ref{jj367}.
The proof of Theorem \ref{tnnn} will be completed with the
inclusion of descendents of odd cohomology in Section \ref{555}.

\subsection{Induction II}

The first half of our induction argument was established
in Lemma \ref{nndd}.
The second half relates the $(0,0)$-tube back to the 
$(0,0)$-cap with an increase in depth.

\begin{lem} The rationality of \label{p45} 
the ${T}$-depth $m$
theory of the $(0,0)$-tube implies the
rationality of $\mathbf{T}$-depth $m+1$ theory of the $(0,0)$-cap.
\end{lem}

\begin{proof}
The result follows from the $\mathbf{T}$-equivariant
localization
formula for the $(0,0)$-cap and the rubber calculus of Section
\ref{papap}.
To illustrate the method, consider first the $m=0$ 
case of Lemma \ref{p45}.

The localization formula for 
$\mathbf{T}$-depth 1 series for the $(0,0)$-cap  is the following:
\begin{multline*}
{\mathsf Z}^{\mathsf{cap}}_{d,\eta} 
\left(   \prod_{j=1}^k \tau_{i_j}([0]) \cdot \tau_{i'_1}([\infty])
\right)^{\mathbf{T}} = 
\\
\sum_{|\mu|=d}
\bW_\mu^{\mathsf{Vert}} \left(\prod_{j=1}^k \tau_{i_j}([0])      \right) \cdot
\bW_\mu^{(0,0)} \cdot
\left( \mathsf{S}^{\tau_{i'_1}\cdot\mu}_\eta+ 
\mathsf{S}^{\mu}_{\eta}(\tau_{i'_1}) \right) \ ,
\end{multline*}
where the rubber terms on the right are 
\begin{eqnarray*}
\mathsf{S}^{\tau_{i'_1}\cdot \mu}_\eta & = &   
\sum_{n\geq d} q^{n}
\left\langle \tau_{i'_1}\cdot \mathsf{P}_\mu \ \left| \ \frac{1}{s_3-\psi_0}  \ \right|\ \CC_\eta 
\right\rangle_{n,d}^{
\sim} ,
 \\
\mathsf{S}^\mu_\eta(\tau_{i'_1}) & = &   
\sum_{n\geq d} q^{n}
\left\langle \mathsf{P}_\mu \ \left| \ \frac{ s_3 \tau_{i'_1}}{s_3-\psi_0}  \ \right|\ \CC_\eta 
\right\rangle_{n,d}^{
\sim}. \\
\end{eqnarray*}

In the first rubber term, $\tau_{i'_1}$ acts
on the boundary condition $P_\mu$ via
\eqref{mrr}. The term arises from the
distribution of the Chern character
of the descendent 
$\tau_{i'_1}([\infty])$
away from the rubber.

The second rubber term simplifies
via  the
topological recursion relation for $\psi_0$ after
writing
\begin{equation}\label{nhhk}
\frac{s_3}{s_3-\psi_0} = 1 + \frac{\psi_0}{s_3-\psi_0}\ 
\end{equation}
and the rubber calculus relation \eqref{dx}. We find
\begin{eqnarray*}
\mathsf{S}^\mu_\eta(\tau_{i'_1}) 
& = &
\sum_{|\widehat{\eta}|=d} \mathsf{S}^\mu_{\widehat{\eta}} 
\cdot \frac{g^{\widehat{\eta}\widehat{\eta}}}{q^d} \cdot 
 {\mathsf Z}^{\mathsf{tube}}_{d,\widehat{\eta},\eta} 
\left(   \tau_{i_1'}([\infty]) \right)^T
 \  -\  \mathsf{S}^{\tau_{i'_1}\cdot \mu}_\eta
.
\end{eqnarray*}
The leading  $1$ on the right side of \eqref{nhhk} corresponds to
the degenerate leading term of $\mathsf{S}^\mu_{\widehat{\eta}}$.
The topological recursion applied to the $\psi_0$ prefactor
of the second term produces the rest of $\mathsf{S}^\mu_{\widehat{\eta}}$.
The superscript $\mathsf{tube}$ refers here to
the $(0,0)$-tube.
The rubber calculus produces the correction
$-\mathsf{S}^{\tau_{i'_1}\cdot \mu}_\eta$.

After reassembling the localization formula, we find
\begin{multline*}
{\mathsf Z}^{\mathsf{cap}}_{d,\eta} 
\left(   \prod_{j=1}^k \tau_{i_j}([0]) \cdot \tau_{i'_1}([\infty])
\right)^{\mathbf{T}} = 
\\
\sum_{|\widehat{\eta}|=d}
{\mathsf Z}^{\mathsf{cap}}_{d,\widehat{\eta}} 
\left(   \prod_{j=1}^k \tau_{i_j}([0]) 
\right)^{\mathbf{T}}
\cdot \frac{g^{\widehat{\eta}\widehat{\eta}}}{q^d} \cdot 
 {\mathsf Z}^{\mathsf{tube}}_{d,\widehat{\eta},\eta} 
\left(   \tau_{i_1'}([\infty]) \right)^T
\end{multline*}
which implies the $m=0$ case of Lemma \ref{p45}.

The above method of expressing the $\mathbf{T}$-depth
$m+1$ theory of the $(0,0)$-cap in terms of
the  $\mathbf{T}$-depth
$0$ theory of the $(0,0)$-cap and the
$T$-depth $m$ theory of the $(0,0)$-tube is valid for
all $m$. 

Consider the $m=1$ case.
The localization formula for 
$\mathbf{T}$-depth 2 series for the $(0,0)$-cap  is the following:
\begin{multline*}
{\mathsf Z}^{\mathsf{cap}}_{d,\eta} 
\left(   \prod_{j=1}^k \tau_{i_j}([0]) \cdot \tau_{i'_1}([\infty])
\tau_{i_2}([\infty])
\right)^{\mathbf{T}} = 
\\
\sum_{|\mu|=d}
\bW_\mu^{\mathsf{Vert}} \left(\prod_{j=1}^k \tau_{i_j}([0])      \right) \cdot
\bW_\mu^{(0,0)} \ \ \ \ \ \ \ \ \ \ \ \ \ \ \ \ \ \ \ \ \ \ 
\ \ \ \ \ \ \  \\
\cdot \left( \mathsf{S}^{\tau_{i'_1}\tau_{i'_2}\cdot\mu}_\eta+ 
\mathsf{S}^{\tau_{i'_1}\cdot \mu}_{\eta}(\tau_{i'_2}) 
+\mathsf{S}^{\tau_{i'_2}\cdot \mu}_{\eta}(\tau_{i'_1})
+\mathsf{S}^{\mu}_{\eta}(\tau_{i'_1}\tau_{i'_2})
\right) \ ,
\end{multline*}
where the rubber terms on the right are 
\begin{eqnarray*}
\mathsf{S}^{\tau_{i'_1}\tau_{i'_2}\cdot \mu}_\eta & = &   
\sum_{n\geq d} q^{n}
\left\langle \tau_{i'_1}\tau_{i'_2}\cdot \mathsf{P}_\mu \ \left| \ \frac{1}{s_3-\psi_0}  \ \right|\ \CC_\eta 
\right\rangle_{n,d}^{
\sim} ,
 \\
\mathsf{S}^{\tau_{i'_1}\cdot \mu}_\eta(\tau_{i'_2}) & = &   
\sum_{n\geq d} q^{n}
\left\langle \tau_{i'_1}\cdot \mathsf{P}_\mu \ \left| \ \frac{ s_3 \tau_{i'_2}}{s_3-\psi_0}  \ \right|\ \CC_\eta 
\right\rangle_{n,d}^{
\sim}, \\
 \\
\mathsf{S}^{\tau_{i'_2}\cdot \mu}_\eta(\tau_{i'_1}) & = &   
\sum_{n\geq d} q^{n}
\left\langle \tau_{i'_2}\cdot \mathsf{P}_\mu \ \left| \ \frac{ s_3 \tau_{i'_1}}{s_3-\psi_0}  \ \right|\ \CC_\eta 
\right\rangle_{n,d}^{
\sim}, \\
 \\
\mathsf{S}^{\mu}_\eta(\tau_{i'_1}\tau_{i'_2}) & = &   
\sum_{n\geq d} q^{n}
\left\langle  \mathsf{P}_\mu \ \left| \ \frac{ s_3^2 \tau_{i'_1}\tau_{i'_2}
}{s_3-\psi_0}  \ \right|\ \CC_\eta 
\right\rangle_{n,d}^{
\sim}. \\
\end{eqnarray*}
Using \eqref{nhhk} and the rubber calculus relation \eqref{dx}, we find
\begin{eqnarray*}
\mathsf{S}^\mu_\eta(\tau_{i'_1}\tau_{i'_2}) 
& = &\ \ 
\sum_{|\widehat{\eta}|=d} \mathsf{S}^\mu_{\widehat{\eta}} 
\cdot \frac{g^{\widehat{\eta}\widehat{\eta}}}{q^d} \cdot 
 {\mathsf Z}^{\mathsf{tube}}_{d,\widehat{\eta},\eta} 
\left(   \tau_{i_1'}([\infty])\cdot \tau_{i'_2}(1) \right)^T
 \  -\  \mathsf{S}^{\tau_{i'_1}\cdot \mu}_\eta(\tau_{i'_2})\\
& & + \sum_{|\widehat{\eta}|=d} \mathsf{S}^\mu_{\widehat{\eta}}(\tau_{i'_2}) 
\cdot \frac{g^{\widehat{\eta}\widehat{\eta}}}{q^d} 
 {\mathsf Z}^{\mathsf{tube}}_{d,\widehat{\eta},\eta} 
\left(   \tau_{i_1'}([\infty]) \right)^T
.
\end{eqnarray*}
As we have seen before,
\begin{eqnarray*}
\mathsf{S}^\mu_{\widehat{\eta}}(\tau_{i'_2}) 
& = &
\sum_{|\widehat{\mu}|=d} \mathsf{S}^\mu_{\widehat{\mu}} 
\cdot \frac{g^{\widehat{\mu}\widehat{\mu}}}{q^d} \cdot 
 {\mathsf Z}^{\mathsf{tube}}_{d,\widehat{\mu},\widehat{\eta}} 
\left(   \tau_{i_2'}([\infty]) \right)^T
 \  -\  \mathsf{S}^{\tau_{i'_2}\cdot \mu}_{\widehat{\eta}}, \\
\mathsf{S}^{\tau_{i'_2}\cdot \mu}_\eta(\tau_{i'_1}) 
& = &
\sum_{|\widehat{\eta}|=d} \mathsf{S}^{\tau_{i'_2}\cdot\mu}_{\widehat{\eta}} 
\cdot \frac{g^{\widehat{\eta}\widehat{\eta}}}{q^d} \cdot 
 {\mathsf Z}^{\mathsf{tube}}_{d,\widehat{\eta},\eta} 
\left(   \tau_{i_1'}([\infty]) \right)^T
 \  -\  \mathsf{S}^{\tau_{i'_1}\tau_{i'_2}\cdot \mu}_\eta
.
\end{eqnarray*}
After adding everything together, we have for $m=1$
the relation:
\begin{multline*}
{\mathsf Z}^{\mathsf{cap}}_{d,\eta} 
\left(   \prod_{j=1}^k \tau_{i_j}([0]) \cdot \prod_{j'=1}^2 
\tau_{i'_{j'}}([\infty])
\right)^{\mathbf{T}} = 
\\
+s_3 \sum_{|\widehat{\eta}|=d}
{\mathsf Z}^{\mathsf{cap}}_{d,\widehat{\eta}} 
\left(   \prod_{j=1}^k \tau_{i_j}([0]) 
\right)^{\mathbf{T}}
\cdot \frac{g^{\widehat{\eta}\widehat{\eta}}}{q^d} \cdot 
 {\mathsf Z}^{\mathsf{tube}}_{d,\widehat{\eta},\eta} 
\left( \tau_{i_1'}([\infty])\cdot  \tau_{i_2'}(1) \right)^T \\
+\sum_{|\widehat{\mu}|,|\widehat{\eta}|=d}
{\mathsf Z}^{\mathsf{cap}}_{d,\widehat{\mu}} 
\left(   \prod_{j=1}^k \tau_{i_j}([0]) 
\right)^{\mathbf{T}}
\cdot \frac{g^{\widehat{\mu}\widehat{\mu}}}{q^d} \cdot 
 {\mathsf Z}^{\mathsf{tube}}_{d,\widehat{\mu},\widehat{\eta}} 
\left(
\tau_{i'_{2}}([\infty])    \right)^T 
\\ 
\cdot \frac{g^{\widehat{\eta}\widehat{\eta}}}{q^d} \cdot 
 {\mathsf Z}^{\mathsf{tube}}_{d,\widehat{\eta},{\eta}} 
\left(
\tau_{i'_{1}}([\infty])    \right)^T 
\ . \ \ \ \ \ \ 
\end{multline*}
We leave the derivation of the parallel formula for general $m$ 
(via elementary bookkeeping) to
the reader.
\end{proof}

An identical argument yields the twisted version of
Lemma \ref{p45} for the $(a_1,a_2)$-cap.

\begin{lem} The rationality of \label{p456} 
the ${T}$-depth $m$
theory of the $(0,0)$-tube implies the
rationality of the $\mathbf{T}$-depth $m+1$ theory of the $(a_1,a_2)$-cap.
\end{lem}

\subsection{Proof of Theorem \ref{cnnn}}
Lemmas \ref{nndd} and \ref{p45}
together provide an induction which results in the
rationality of the $\mathbf{T}$-depth $m$ theory of
the $(0,0)$-cap for all $m$.
Since the classes of the $\mathbf{T}$-fixed points $0,\infty \in \Pp$
generate $H_{\mathbf{T}}^*(\Pp, \mathbb{Z})$
after localization, all partition functions
$$
{\mathsf Z}^{\mathsf{cap}}_{d,\eta} 
\left(   \prod_{j=1}^k \tau_{i_j}(\gamma_{j})
\right)^{\mathbf{T}},\ \ \ \ \gamma_{j}\in H^*_{\mathbf{T}}(\Pp,\mathbb{Z})$$
are Laurent series
in $q$ of rational functions in $\mathbb{Q}(q,s_1,s_2,s_3)$.
\qed

\subsection{Proof of Proposition \ref{pnnn}}  \label{jj367}
Using Lemma \ref{p456}, we obtain the extension of
Theorem \ref{cnnn} to twisted $(a_1,a_2)$-caps.
\begin{prop}
\label{qnnn} 
For $\gamma_{j}\in H^*_{\mathbf{T}}(\Pp,\mathbb{Z})$, the descendent series
$$
{\mathsf Z}^{(a_1,a_2)}_{d,\eta} 
\left(   \prod_{j=1}^k \tau_{i_j}(\gamma_{j})
\right)^{\mathbf{T}}$$
of the $(a_1,a_2)$-cap
is the 
Laurent expansion in $q$ of a rational function in 
$\mathbb{Q}(q,s_1,s_2,s_3)$.
\end{prop}

By taking the $s_3=0$ specialization of Proposition
\ref{qnnn}, we obtain the rationality of the $T$-depth $m$
 theory of the $(a_1,a_2)$-cap for all $m$.
Proposition \ref{pnnn} then follows from Lemma \ref{rtt5}. \qed

\subsection{$\mathbf{T}$-equivariant tubes}
The $(a_1,a_2)$-tube is the total space of
$$\cO_{\Pp}(a_1) \oplus \cO_{\Pp}(a_2) \rightarrow \Pp$$
relative to the fibers over both $0,\infty \in \Pp$.
We lift the $\com^*$-action on $\Pp$ to $\cO_{\Pp}(a_i)$ 
with fiber weights $0$ and $a_is_3$ over $0,\infty\in \PP^1$.
The 2-dimensional torus $T$ acts on the $(a_1,a_2)$-tube
by scaling the line summands, so
we obtain a $\mathbf{T}$-action on the $(a_1,a_2)$-tube.

\begin{prop}
\label{qqnnn} 
For $\gamma_{j}\in H^*_{\mathbf{T}}(\Pp,\mathbb{Z})$, the 
descendent series
$$
{\mathsf Z}^{(a_1,a_2)}_{d,\eta_1\eta_2} 
\left(   \prod_{j=1}^k \tau_{i_j}(\gamma_{j})
\right)^{\mathbf{T}}$$
of the $(a_1,a_2)$-tube
is the 
Laurent expansion in $q$ of a rational function in 
$\mathbb{Q}(q,s_1,s_2,s_3)$.
\end{prop}

\begin{proof}
Consider the descendent series
\begin{equation}\label{pw4}
{\mathsf Z}^{(a_1,a_2)}_{d,\eta_2} 
\left(  L  \prod_{j'=1}^{k'} \tau_{i'_{j'}} (\mathsf{1})\ \prod_{j=1}^k \tau_{i_j}([\infty])
\right)^{\mathbf{T}}
\end{equation}
of the $(a_1,a_2)$-cap where $L$ is a monomial in the descendents of 
$[0]$.
The $(a_1,a_2)$-cap admits a $\mathbf{T}$-equivariant
degeneration to a standard $(0,0)$-cap and an $(a_1,a_2)$-tube by
bubbling off $0\in \Pp$.
The insertions $\tau_{i_j}([0])$ of $L$ are sent $\mathbf{T}$-equivariantly
to the non-relative point of the $(0,0)$-cap.
Since \eqref{pw4} is rational by Proposition \ref{qnnn} and the
matix $M_d$ of Lemma \ref{rtt5} is full rank, 
the rationality of 
$$
{\mathsf Z}^{(a_1,a_2)}_{d,\eta_1\eta_2} 
\left( \prod_{j'=1}^{k'} \tau_{i'_{j'}} (\mathsf{1})\ \prod_{j=1}^k \tau_{i_j}([\infty]) 
\right)^{\mathbf{T}}$$
follows by induction on $k'$ from the degeneration formula.
The classes $\mathsf{1}$ and $[\infty]$ generate $H_{\mathbf{T}}^*(\Pp,\mathbb{Z})$
after localization.
\end{proof}

\section{Descendents of odd cohomology} \label{555}

\subsection{Reduction to $(0,0)$}
Let $N/S$ be the relative
geometry of 
a split rank 2 bundle on a nonsingular projective curve $C$ of genus $g$.
Let 
$$\alpha_1, \ldots, \alpha_g, \beta_1, \ldots, \beta_g \in H^1(C,\mathbb{Z})$$
be a standard symplectic basis of the odd cohomology of $C$.
Proposition \ref{pnnn} establishes Theorem \ref{tnnn}
in case only the descendents of the even classes
 $\mathsf{1},\mathsf{p}\in H^*(C,\mathbb{Z})$ are present.
The descendents of $\alpha_i$ and $\beta_j$
will now be considered.

The relative geometry $N/S$ may be $T$-equivariantly degenerated to
$$\cO_C \oplus \cO_C \rightarrow C$$
and an $(a_1,a_2)$-cap. The relative points and
the  
descendents $\tau_k(\alpha_i)$ and $\tau_k(\beta_j)$ in the
integrand remain on $C$. Since the rationality of the
$T$-equivariant descendent theory  of the $(a_1,a_2)$-cap
has been proven, we may restrict our study of the descendents
of odd cohomology to the $(0,0)$ relative geometry over $C$.

\subsection{Proof of Theorem \ref{tnnn}}
The full descendent theories of $(0,0)$ relative geometries of 
curves $C$ 
are uniquely determined by the even descendent theories of $(0,0)$ relative geometries
by the following four properties:
\begin{enumerate}
\item[(i)] Algebraicity of the virtual class,
\item[(ii)] Degeneration formulas for the relative theory in the
presence of odd cohomology,
 \item[(iii)] Monodromy invariance of the relative theory,
\item[(iv)] Elliptic vanishing relations.
\end{enumerate}
The properties (i)-(iv) were used in \cite{vir} to determine the
full relative Gromov-Witten descendents of target curves in terms
of the descendents of even classes.

The results of Section 5 of \cite{vir}  are entirely formal and
apply verbatim to the descendent theory of $(0,0)$ relative geometries of 
curves. Moreover, the rationality of the even theory implies
the rationality of the full descendent theory. \qed

\section{Denominators} \label{ennd}

\subsection{Summary} We prove the
denominator claims of Conjecture \ref{222}
when only descendents of $\mathsf{1}$ and $\mathsf{p}$ 
are present. 

\begin{thm}
\label{2222}
If only descendents of even cohomology are considered, 
the denominators of the degree $d$ descendent partition functions
$\ZZ$ of Theorems \ref{onnn}, \ref{tnnn}, and \ref{cnnn}
are  products of factors of the form $q^k$ and
$$1-(-q)^r$$
for $1\leq r \leq d$.
\end{thm}

Theorem \ref{2222} is proven by  carefully tracing
the denominators through the proofs of Theorems \ref{onnn}-\ref{cnnn}.
When the descendents of odd cohomology are included,
the strategy of Section 5 of \cite{vir} requires
matrix inversions{\footnote{Specifically,
the matrix associated to Lemma 5.6 of \cite{vir}
has an inverse with denominators we cannot
at present constrain.}} for which we can not control the
denominators.

Theorem \ref{2222} is new even when {\em no} descendents
are present. For the trivial bundle
$$N = \OO_{\mathbb{P}^1} \oplus \OO_{\mathbb{P}^1} \rightarrow
\mathbb{P}^1\ ,$$
the $T$-equivariant partition $\mathsf{Z}^{N/S}_{d,\eta^1,\eta^2,\eta^3}$
of Theorem \ref{tnnn}
is (up to $q$ shifts) equal to the 3-point
function $$\langle \eta^1, \eta^2,\eta^3 \rangle$$
in the quantum cohomology of the Hilbert scheme of points
of $\mathbb{C}^2$, see \cite{hilb1,lcdt}.

\vspace{10pt}
\noindent{\bf Corollary.} {\em  The  3-point functions in the $T$-equivariant
quantum cohomology of $\text{Hilb}(\mathbb{C}^2,d)$
have possible poles in -q only at the $r^{th}$ roots of
unity for $r$ at most $d$.}
\vspace{10pt}

\begin{proof}
By Theorem \ref{2222}, we see the possible poles in $-q$
of the 3-point
functions are at $0$ and the $r^{th}$ roots of
unity for $r$ at most $d$.
By definition, the 3-point functions have
no poles at 0.
\end{proof}

\subsection{Denominators for Proposition \ref{cttt}}
We follow here the notation used in the proof of Proposition 
\ref{cttt} in Section \ref{333}.

The matrix $\mathsf{S}_\eta^\mu$ is a fundamental solution of a
linear differential equation with singularities only at
0 and $r^{th}$ roots of unity for $r$ at most $d$, see
 \cite{hilb2}. Hence, the poles in $-q$ of the evaluation 
$$\mathsf{S}_\eta^\mu
|_{s_3=\frac{1}{a}(s_1+s_2)}$$ 
can occur only at $0$ and  $r^{th}$ roots of unity for $r$ at most $d$.
The denominator claim of Theorem \ref{2222} for 
Proposition \ref{cttt} then follows directly from the proof
in Section \ref{ggtt2}. 

While only the rationality  of Theorem \ref{canpole} is needed 
in the proof of Proposition \ref{cttt}, the much stronger Laurent
{polynomiality} 
of Theorem \ref{canpole} is used here.

\subsection{Denominators for $T$-equivariant stationary theory}
Consider the denominators of
\begin{equation*}
{\mathsf Z}^{N/S}_{d,\eta^1,\dots,\eta^r} 
\left( \prod_{j=1}^k \tau_{i_j}(\mathsf{p})   
\right)^T \ .
\end{equation*}
The denominator result for the $T$-equivariant stationary theory of
the
$(0,0)$-cap is obtained from the denominator result for
Proposition \ref{cttt} by the specialization $s_3=0$.
By degenerating all the descendents $\tau_{i_j}(\mathsf{p})$
on to a $(0,0)$-cap, we need only study the denominators
of $T$-equivariant 
partition functions ${\mathsf Z}^{N/S}_{d,\eta^1,\dots,\eta^r}$
with no descendent insertions.

The denominator result for 
the $T$-equivariant $(a,b)$-tube with no descendents
is again a consequence
of the study of the fundamental solution in \cite{hilb2}.
By repeated degenerations (using the $(a,b)$-tube for the
twists in $N$), we need only study the 
 denominators
of $T$-equivariant 
partition functions ${\mathsf Z}^{(0,0)}_{d,\eta^1,\eta^2,\eta^3}$
with 3 relative insertions.

\subsection{Relative/descendent correspondence}

Relative conditions in the theory of local curves
were exchanged for descendents
in the proof of Lemma \ref{rtt5}.
For the denominator result for
${\mathsf Z}^{(0,0)}_{d,\eta^1,\eta^2,\eta^3}$, we require a
more efficient correspondence.

\begin{prop}\label{gbb5}
Let $d>0$ be an integer.
The square matrix with coefficients
\begin{equation}\label{gredd}
\mathsf{Z}^{\mathsf{cap}}_{d,\lambda}\left( \tau_{\mu_1-1}([0])
\cdots \tau_{\mu_{\ell(\mu)}-1}([0]) \right)^T
\end{equation}
as  $\lambda$ and $\mu$ vary among
partitions of $d$ 
\begin{enumerate}
\item[(i)] is triangular with respect to the partial ordering
           by length, 
\item[(ii)] has
diagonal entries given
           by monomials in $q$,
\item[(iii)] and is of
maximal rank.
\end{enumerate}
\end{prop}

\begin{proof}
The Proposition follows from the 
results of Section 4.6 of \cite{lcdt} applied to
the theory of stable pairs.
Our relative conditions $\lambda$ are defined
with identity weights in the $T$-equivariant cohomology
of $\mathbb{C}^2$. For the proof, we weight all the parts of
$\lambda$ with he $T$-equivariant class of the origin in
$\mathbb{C}^2$. Then, by compactness and dimension
constraints, the triangularity of the matrix is
immediate for partitions of different lengths.
On the diagonal, the expected dimension of
the integrals are 0. Using the compactification
\begin{equation}\label{jttm}
\mathbb{C}^2 \times \mathbb{P}^1 \subset \mathbb{P}^2 \times
\mathbb{P}^1
\end{equation}
as in Section 4.6 of \cite{lcdt}, we 
obtain the triangularity of equal length partitions.

Consider the Hilbert scheme of points ${\text{Hilb}}(\com^2,d)$ of the
plane.
Let
${\mathbb{F}}$ be the universal quotient sheaf on 
$$\text{Hilb}(\com^2,d) \times \com^2,$$ and define the 
descendent{\footnote{The Chern character of $\mathbb{F}$
is properly supported over ${\text{Hilb}}(\com^2,d)$.} }
\begin{equation*}
\tau_k=\pi_*\Big( {\text{ch}}_{2+k}({\mathbb{F}})\Big) 
\in A^{k}({\text{Hilb}}(\com^2,d),\mathbb{Q})
\end{equation*}
as before \eqref{mrr}.
Using the compactification \eqref{jttm}, we
reduce the calculation of the diagonal entries to
the pairing
\begin{equation}\label{appp}
 s_1s_2\ \Big\langle \tau_{c-1} \ \Big| \ (c) \Big\rangle _
{{\text{Hilb}}(\com^2,d)}
= \frac{1}{c!}\ 
\end{equation}
which appears in \cite{parttwo}.

We conclude the diagonal entries do not vanish.
The diagonal entries are monomial in $q$ by the usual
vanishing obtained by the holomorphic symplectic form on $\mathbb{C}^2$.
\end{proof}

The denominator result holds for the
nonvanishing entries of the correspondence
matrix \eqref{gredd}. Since the matrix is triangular
with monomials in $q$ on the diagonal, the denominator
result holds for the {\em inverse} matrix.

We can now establish the denominator result for
the $T$-equivariant 3-point function
${\mathsf Z}^{(0,0)}_{d,\eta^1,\eta^2,\eta^3}$.
We start with the descendent series
\begin{equation}\label{jj87}
{\mathsf{Z}}^{0,0}_{d,\eta^3}\left( 
\tau_{\mu_1-1}(\mathsf{p})
\cdots \tau_{\mu_{\ell(\mu)}-1}(\mathsf{p})\cdot \tau_{\widehat{\mu}_1-1}(\mathsf{p})
\cdots \tau_{\widehat{\mu}_{\ell(\widehat{\mu})}-1}(\mathsf{p})\right)
\end{equation}
for partitions $\mu$ and $\widehat{\mu}$ of $d$.
The denominator result holds for all
series \eqref{jj87}. By bubbling all the
descendents $\tau_{\mu_i-1}(\mathsf{p})$ off of the point $0\in \mathbb{P}^1$
and
bubbling all the 
descendents $\tau_{\widehat{\mu}_i-1}(\mathsf{p})$ off of the point $1\in \mathbb{P}^1$, we conclude the
denominator result for ${\mathsf Z}^{(0,0)}_{d,\eta^1,\eta^2,\eta^3}$
from the denominator result for the 
inverse of the correspondence matrix \eqref{gredd}.

\subsection{Denominators for Theorems \ref{tnnn}-\ref{cnnn}}
The denominator result for Theorem \ref{cnnn} is obtained by following
the proof given in Sections \ref{depp}-\ref{444}. An important point is
to replace the matrix $M_d$ appearing in the proof of
Lemma \ref{rtt5} with the correspondence matrix \eqref{gredd}. The required
matrix inversion then keeps the denominator form.
The rest of the proof of Theorem \ref{cnnn} respects the
denominators.

Proposition \ref{pnnn} is the statement of Theorem \ref{tnnn} for descendents
of even cohomology. Again, the proof respects the denominators.
The proof of Theorem \ref{2222} is complete.\qed

\vspace{+8 pt}
\noindent
Department of Mathematics\\
Princeton University\\
rahulp@math.princeton.edu

\vspace{+8 pt}
\noindent
Department of Mathematics\\
Princeton University\\
apixton@math.princeton.edu


\begin{thebibliography}{MNOP2}






\bibitem{Beh} K. Behrend, {\em Donaldson-Thomas invariants via
microlocal geometry}, math.AG/0507523.

\bibitem{BehFan} K.~Behrend and B.~Fantechi,
{\em The intrinsic normal cone}, Invent.\ Math.\
{\bf 128} (1997), 45--88.


\bibitem{bridge} T. Bridgeland, {\em Hall algebras and curve-counting
invariants}, arXiv:1002.4372.




\bibitem{BryanP}
J.~Bryan and R.~Pandharipande,
\newblock {\em The local {G}romov-{W}itten theory of curves}, 
JAMS {\bf 21} (2008), 101--136.




\bibitem{DT}
S.~K. Donaldson and R.~P. Thomas,
\newblock {\em Gauge theory in higher dimensions}.
\newblock In {\em The geometric universe (Oxford, 1996)},  31--47. Oxford Univ.
  Press, Oxford, 1998.

\bibitem{FP} C. Faber and R. Pandharipande, {\em Hodge integrals and
Gromov-Witten theory}, Invent. Math. {\bf 139} (2000), 173--199.


\bibitem{GraberP}
T.~Graber and R.~Pandharipande,
\newblock {\em Localization of virtual classes}, Invent. Math., {\bf 135},
  487--518, 1999.



\bibitem{HLShaves}
D.~Huybrechts and M.~Lehn,
\newblock {\em The geometry of moduli spaces of shaves}.
\newblock Aspects of Mathematics, E31. Friedr. Vieweg \& Sohn, Braunschweig,
  1997.


\bibitem{joy} D. Joyce and Y. Song, {\em A theory of generalized
Donaldson-Thomas invariants}, arXiv:0810.5645.



\bibitem{LPPairs1}
J.~Le~Potier,
\newblock {\em Syst\`emes coh\'erents et structures de niveau}, Ast\'erisque,
  {\bf 214}, 1993.

\bibitem{LiTian} J.~Li and G.~Tian, 
{\em Virtual moduli cycles and Gromov-Witten invariants
of algebraic varieties}, 
J. AMS {\bf 11} (1998), 119--174.




\bibitem{liwu} J. Li and B. Wu, {\em Degeneration of Donaldson-Thomas
invariants}, preprint 2009.





\bibitem{moop}
D.~Maulik, A. ~Oblomkov, A.~Okounkov, and R.~Pandharipande,
\newblock {\em The Gromov-{W}itten/{D}onaldson-{T}homas correspondence
for toric 3-folds}, Invent. Math. (to appear).
arXiv:0809.3976.

\bibitem{MNOP1}
D.~Maulik, N.~Nekrasov, A.~Okounkov, and R.~Pandharipande,
\newblock {\em Gromov-{W}itten theory and {D}onaldson-{T}homas theory. {I}},
  Compos. Math. {\bf 142} (2006), 1263--1285.


\bibitem{MNOP2}
D.~Maulik, N.~Nekrasov, A.~Okounkov, and R.~Pandharipande,
\newblock {\em Gromov-{W}itten theory and {D}onaldson-{T}homas theory. {II}},
  Compos. Math. {\bf 142} (2006), 1286--1304.

\bibitem{mptop}
D.~Maulik and R.~Pandharipande,
\newblock {\em A topological view of Gromov-Witten theory},
  Topology {\bf 45} (2006), 887--918.




\bibitem{mpt}
D. Maulik, R. Pandharipande, R. Thomas, {\em Curves on $K3$ surfaces
and modular forms}, J. Topology {\bf 3} (2010), 937--996.


\bibitem{vir}
A.~Okounkov and R.~Pandharipande,
\newblock {\em Virsoro constraints for target curves},
Invent. Math. {\bf 163} (2006), 47--108.

\bibitem{hilb1}
A.~Okounkov and R.~Pandharipande,
\newblock {\em The quantum cohomology of the Hilbert
scheme of points of the plane},
Invent. Math. {\bf 179} (2010), 523--557.



\bibitem{lcdt}
A.~Okounkov and R.~Pandharipande,
\newblock {\em The local Donaldson-Thomas theory of curves},
Geometry and Topology {\bf 14} (2010), 1503--1567.

\bibitem{hilb2}
A.~Okounkov and R.~Pandharipande,
\newblock {\em The quantum differential equation of the Hilbert
scheme of points of the plane},
Invent. Math. {\bf 178} (2010), 523--557.


\bibitem{parttwo}
R.~Pandharipande and A.~Pixton,
\newblock {\em Descendents on local curves: stationary theory}, 
arXiv:1109.1258.

\bibitem{part3}
R.~Pandharipande and A.~Pixton,
\newblock {\em Descendent theory for stable pairs on toric 3-folds}, 
arXiv:1011.4054.


\bibitem{pt}
R.~Pandharipande and R.~P. Thomas,
\newblock {\em Curve counting via stable pairs in the derived
category}, Invent Math. {\bf 178} (2009), 407 -- 447.


\bibitem{pt2}
R.~Pandharipande and R.~P. Thomas,
\newblock {\em The 3-fold vertex via stable pairs}, Geometry and Topology
{\bf 13} (2009), 1835--1876.

\bibitem{pt3}
R.~Pandharipande and R.~P. Thomas,
\newblock {\em Stable pairs and BPS invariants}, JAMS {\bf 23}
(2010), 267--297.



\bibitem{Thomas}
R.~P. Thomas,
\newblock {\em A holomorphic {C}asson invariant for {C}alabi-{Y}au 3-folds, and
  bundles on {$K3$} fibrations}, J. Differential Geom., {\bf 54}, 367--438,
  2000.


\bibitem{toda}
Y. Toda, {\em Generating functions of stable pairs invariants
via wall-crossings in derived categories}, arXiv:0806.0062.


\end{thebibliography}
\end{document}